\documentclass[reqno,11pt]{amsart}

\usepackage{palatino}

\usepackage{amsmath}
\usepackage{amsfonts,amssymb,amsmath,latexsym,wasysym,mathrsfs,bbm,stmaryrd}
\usepackage[all]{xy}
\usepackage{color}
\usepackage{epsfig}

\def\R{\mathbb R}
\def\C{\mathbb C}

\def\N{\mathbb N}

\def\1{\mathbbm 1}

\def\e{\epsilon}

\def\s{\sigma}
\def\p{\varphi}
\def\k{\kappa}
\def\l{\lambda}

\def\nn{\mx{n}}
\def\mm{\mx{m}}

\def\supp{\text{supp}\,}

\newcommand{\ff}{\varphi}

\newcommand{\interior}[1]{\raise0.2ex\hbox{$\displaystyle{\mathop{#1}^{\circ}}$}}
\newcommand{\mx}[1]{\mathbf{#1}}

\renewcommand\phi{\varphi}

\newtheorem{theorem}{Theorem}[section]
\newtheorem*{theorem*}{Theorem}

\newtheorem{proposition}[theorem]{Proposition}
\newtheorem{definition}[theorem]{Definition}
\newtheorem{corollary}[theorem]{Corollary}

\newtheorem{lemma}[theorem]{Lemma}
\newtheorem{notation}[theorem]{Notation}
\numberwithin{equation}{section} 

\theoremstyle{remark}
\newtheorem*{remark*}{Remark}
\newtheorem{remark}[theorem]{Remark}

\theoremstyle{remark}
\newtheorem{example}[theorem]{Example}

\long\def\symbolfootnote[#1]#2{\begingroup%
\def\thefootnote{\fnsymbol{footnote}}\footnote[#1]{#2}\endgroup}

\begin{document}

\title{$\mathscr{R}$-diagonal dilation semigroups}
\author{Todd Kemp}
\address{Department of Mathematics, MIT \\ 2-172, 77 Massachusetts Avenue, Cambridge, MA \; 02139}
\email{tkemp@math.mit.edu}

\begin{abstract} This paper addresses extensions of the complex Ornstein-Uhlenbeck semigroup to operator algebras in free probability theory.  If $a_1,\ldots,a_k$ are $\ast$-free $\mathscr{R}$-diagonal operators in a $\mathrm{II}_1$ factor, then $D_t(a_{i_1}\cdots a_{i_n}) = e^{-nt} a_{i_1}\cdots a_{i_n}$ defines a dilation semigroup on the non-self-adjoint operator algebra generated by $a_1,\ldots,a_k$.  We show that $D_t$ extends (in two different ways) to a semigroup of completely positive maps on the von Neumann algebra generated by $a_1,\ldots,a_k$.  Moreover, we show that $D_t$ satisfies an optimal ultracontractive property: $\|D_t\colon L^2\to L^\infty\| \sim t^{-1}$ for small $t>0$. \end{abstract}

\maketitle

\symbolfootnote[0]{This work was partially supported by NSF Grant DMS-0701162.}

\vspace{-0.4in}

\section{Introduction and Background}

This paper is a sequel to \cite{Kemp Speicher}, in which the authors discuss an important norm inequality (the Haagerup inequality) in the context of certain non-normal operators ($\mathscr{R}$-diagonal operators) in free probability.  The motivation for these papers comes from the classical Ornstein-Uhlenbeck semigroup in Gaussian spaces, which we will briefly recall now.  Let $\gamma_d$ denote Gauss measure on $\R^d$ (the standard $n$-dimensional normal law).  The Ornstein-Uhlenbeck semigroup $U_t$ is the $\mathscr{C}_0$ Markov semigroup on $L^2(\R^d,\gamma_d)$ associated to the Dirichlet form of the measure $(f,f)\mapsto \int |\nabla f|^2\,d\gamma$.  Its infinitesimal generator $N$, called the Ornstein-Uhlenbeck operator or number operator, is given by $Nf(\mx{x}) = -\Delta f(\mx{x}) + \mx{x}\cdot\nabla f(\mx{x})$.  The O--U semigroup can be expressed as a multiplier semigroup (with integer eigenvalues) in terms of tensor products of Hermite polynomials.

The space $L^2(\R^{2d},\gamma_{2d})$ contains many holomorphic functions; for example, all monomials $z^{\mx{n}} = z_1^{n_1}\cdots z_d^{n_d}$ with $\mx{n} = (n_1,\ldots,n_d)$.  The space of holomorphic $L^2$-functions, $L^2_{hol}(\C^d,\gamma_{2d})$, is a Hilbert space that reduces the O--U semigroup.  Since $\Delta h = 0$ for $h\in L^2_{hol}(\C^d,\gamma_{2d})$, the restriction of the number operator is $Nh(\mx{z}) = \mx{z}\cdot \nabla h(\mx{z})$ which is sometimes called the Euler operator, the infinitesimal generator of dilations.  As a result, for holomorphic $h$ it follows that $U_t h (\mx{z}) = h(e^{-t}\mx{z})$.  In terms of monomials, $U_t (z^{\mx{n}}) = e^{-|\mx{n}|t} z^{\mx{n}}$, where $|\mx{n}| = n_1+\cdots+n_d$.  This simpler action has many important consequences for norm estimates (in particular hypercontractivity) in such spaces; see \cite{Janson,Gross 3, GKLZ}. 

There is a natural analogue of the complex variable $z$ in free probability.  Let $s,s'$ be free semicircular operators in a $\mathrm{II}_1$ factor; these are natural analogues of independent normal random variables.  (For basics on free probability, see the book \cite{Nica Speicher Book}.)  Then $c = (s + is')/\sqrt{2}$ is Voiculescu's {\em circular} operator.  Aside from its obvious analogous appearance to a complex standard normal random variable in $L^2_{hol}(\C,\gamma_2)$, it can also be thought of as a limit $N\to\infty$ of the $N\times N$ Ginibre ensemble of matrices with all independent complex normal entries (of variance $1/2N$).  To mimic the random vector $\mx{z}=(z_1,\ldots,z_d)$ one can take $\ast$-free circular operators $c_1,\ldots,c_d$, and $d$ can even be infinite.  The analogue of the O--U semigroup is then simply $D_t(c_{i_1}\cdots c_{i_n}) = e^{-nt} c_{i_1}\cdots c_{i_n}$.  In this context, the same kinds of strong norm estimates referred to above are discussed in the author's paper \cite{Kemp}.  This dilation semigroup is actually a restriction of a semigroup of completely positive maps, the {\em free O--U semigroup}, on the full von Neumann algebra $W^\ast(c_1,\ldots,c_d)$, as considered in \cite{Biane 1, Biane 2}.

Circular operators are the prime examples of {\em $\mathscr{R}$-diagonal operators}.  Introduced in \cite{Nica Speicher Fields Paper}, $\mathscr{R}$-diagonal operators form a large class of non-self-adjoint operators that all have rotationally-invariant distributions in a strong sense.  They have played important roles in a number of different problems in free probability; see \cite{Haagerup 2, NSS, Sniady Speicher}).  In \cite{Kemp Speicher}, the authors proved a strong form of a norm inequality (the Haagerup inequality) for an $L^2_{hol}$-space in the context of $\mathscr{R}$-diagonal operators.  A corollary to the estimates therein is a norm inequality (ultracontractivity) for a dilation semigroup akin to the one above: if $a_1,\ldots,a_d$ are $\mathscr{R}$-diagonal and $\ast$-free, then $D_t$ is defined by $D_t(a_{i_1}\cdots a_{i_n}) = e^{-nt} a_{i_1}\cdots a_{i_n}$.  Phillippe Biane asked the author if this dilation semigroup generally has a completely positive extension to the von Neumann algebra $W^\ast(a_1,\ldots,a_d)$, as it does in the special case that each operator $a_j$ is circular.  The first main theorem (proved in Section \ref{section 2}) of this paper answers that question in the affirmative.

\begin{theorem} \label{theorem 1} Let $d\in\{1,2,\ldots,\infty\}$, and let $a_1,\ldots,a_d$ be $\ast$-free $\mathscr{R}$-diagonal operators.  Then the dilation semigroup $D_t$, defined on the algebra generated by $a_1,\ldots,a_d$ (and {\em not} $a_1^\ast,\ldots,a_d^\ast$) by $D_t(a_{i_1}\cdots a_{i_n}) = e^{-nt} a_{i_1}\cdots a_{i_n}$, has a completely positive extension to $W^\ast(a_1,\ldots,a_d)$, given by
\[ D_t(a_{i_1}^{\e_1} \cdots a_{i_n}^{\e_n}) = e^{-|\e_1 + \cdots + \e_n|t} a_{i_1}^{\e_1} \cdots a_{i_n}^{\e_n} \]
where $\e_j\in\{1,-1\}$ and $a_j^{-1}$ is interpreted as $a_j^\ast$.
\end{theorem}

This extension is precisely the kind of semigroup considered in the case of Haar unitary generators in \cite{JX} and \cite{JLX}.  (For example, if $d=1$ and the single generator is unitary $u$, then $D_t$ is simply the homomorphism generated by $u\mapsto e^{-t}u$.)  However, in the circular case, this extension is not the natural one (it is much simpler than Biane's free O--U semigroup which is diagonalized by free products of Techebyshev polynomials).  In fact, the above completely positive extension is generally non-unique: in the circular case, the dilation semigroup also has the free O--U semigroup as a CP extension.  In Section \ref{section 2}, we will provide a framework for more natural CP extensions (including the O--U semigroup), depending on a Markovian character of the distribution of the absolute value of the $\mathscr{R}$-diagonal generators.

\medskip

The second half of this paper concerns $L^p$-bounds of the non-selfadjoint semigroup $D_t$ of Theorem \ref{theorem 1}.  In the classical context of the O--U semigroup $U_t$ acting in $L^2(\gamma)$ or restricted to $L^2_{hol}(\gamma)$, for any finite $p>2$ the map $U_t$ is bounded into $L^p$ for sufficiently large $t$.  However, it is never bounded into $L^\infty$. This is not the case in the free analogue.  Let us fix some notation.

\begin{notation} \label{notation L2hol} Let $A=\{a_1,\ldots, a_d\}$ be $\ast$-free $\mathscr{R}$-diagonal operators in a $\mathrm{II}_1$-factor with trace $\ff$.  Denote by $L^2_{hol}(a_1,\ldots,a_d)$ the Hilbert subspace of $L^2(W^\ast(A),\ff)$ generated by the (non-$\ast$) algebra generated by $A$. \end{notation}

In \cite{Kemp Speicher} we proved the following.

\begin{theorem*}[Theorem 5.4 in \cite{Kemp Speicher}]  Suppose that $a_1,\ldots,a_d$ are $\ast$-free $\mathscr{R}$-diagonal operators satisfying $C = sup_{1\le j\le d} \|a_j\|/\|a_j\|_2 < \infty$ (e.g.\ if $d$ is finite).  Then for $t>0$,
\[ \|D_t\colon L^2_{hol}(a_1,\ldots,a_d)\to W^\ast(a_1,\ldots,a_d)\| \le 515\sqrt{e}\,C^2\, t^{-1}. \]
\end{theorem*}
(In \cite{Kemp Speicher} this Theorem is stated only in the case that $a_1,\ldots,a_d$ are identically--distributed, but a glance at the proof of Theorem 1.3 in \cite{Kemp Speicher} shows that the the theorem was actually proved in the generality stated above.)  The following theorem shows that this ultracontractive bound is, in fact, sharp.

\begin{theorem} \label{theorem 3} Let $a_1,\ldots,a_d$ be $\ast$-free $\mathscr{R}$-diagonal operators, at least one of which is not a scalar multiple of a Haar unitary.  Then there are constants $\alpha,\beta>0$ such that, for $0<t<1$,
\[ \alpha\,t^{-1}\; \le \|D_t\colon L^2_{hol}(a_1,\ldots,a_d)\to W^\ast(a_1,\ldots,a_d)\| \le \;  \beta\, t^{-1}. \]
Moreover, this bound is achieved on the algebra generated by a single non-Haar-unitary $a_j$.
\end{theorem}
In fact, in Section 3 we will give sharp bounds for the action of $D_t$ from $L^2_{hol}$ to $L^p$ for any even integer $p\ge 2$, at least in the case that the generator $a_j$ has non-negative cumulants.  This is not quite enough to yield the bound of Theorem \ref{theorem 3}, but only the infinitesimally smaller lower-bound $t^{-1+\e}$ for any $\e>0$.  The full theorem is proved instead using a clever $L^4$ estimate suggested by Haagerup.

\medskip

In what follows, we will provide a minimum of technical background on the free probabilistic tools used when needed.  We suggest that readers consult the ``Free Probability Primer" (Section 2) in \cite{Kemp Speicher}, and the excellent book \cite{Nica Speicher Book} for further details.

\begin{remark} A note on constants.  The $\|D_t\colon L^2\to L^p\|$--estimates considered in this paper are of interest for the order of magnitude blow-up as $t\to 0$.  Multiplicative constants will be largely ignored.  As such, the symbols $\alpha,\beta$ will sometimes be used to represent arbitrary positive constants, and so some equations may seemingly imply false relations like $2\alpha = \alpha^2 = \alpha/\alpha = \sqrt{\alpha} = \alpha$.  The author hopes this will not cause the reader any undue stress.
\end{remark}

\section{Completely Positive Extensions}\label{section 2}

\subsection{Preliminaries} We begin with a few basic facts about $\mathscr{R}$-diagonal operators.  Fix a $\mathrm{II}_1$-factor $\mathscr{A}$ with trace $\ff$. $NC(n)$ denotes the lattice of non-crossing partitions of the set $\{1,\ldots,n\}$. The {\bf free cumulants} $\{\kappa_\pi\,;\,\pi\in\bigsqcup_n NC(n)\}$ relative to $\ff$ are multilinear functionals $\mathscr{A}^n\to\C$, given by the M\"obius inversion formula
\begin{equation}\label{free cumulants} \k_\pi[a_1,\ldots,a_n] = \sum_{\s\le\pi} \ff_\s[a_1,\ldots,a_n]\,\mathrm{Moeb}(\s,\p), \end{equation}
where $\mathrm{Moeb}$ is the M\"obius function of the lattice $NC(n)$, and $\ff_\s[a_1,\ldots,a_n]$ is the product of moments of the arguments corresponding to the partition $\pi$: if the blocks of $\s$ are $\{V_1,\ldots,V_r\}$, then $\ff_\s = \ff_{V_1}\cdots \ff_{V_r}$, where, if $V = \{i_1<\cdots<i_k\}$,
\[ \ff_V[a_1,\ldots,a_n] = \ff(a_{i_1}\cdots a_{i_k}). \]
For example, if $\pi=\{\{1,4\},\{2,5\},\{3\}\}$ then $\ff_\pi[a_1,\dots,a_5] = \ff(a_1a_4)\,\ff(a_2a_5)\,\ff(a_3)$.  Let $\k_n$ stand for $\k_{1_n}$ where $1_n$ is the one block partition $\{1,\ldots,n\}$.  Then, for example, $\k_1[a] = \ff(a)$ is the mean, while $\k_2[a,b] = \phi(ab) - \phi(a)\phi(b)$ is the covariance.  It is important to note that the functionals $\k_\pi$ also share the same factorization property as the functionals $\ff_\pi$: if $\pi = \{V_1,\ldots,V_r\}$ then $\k_\pi = \k_{V_1}\cdots \k_{V_r}$, where, if $V = \{i_1<\cdots<i_k\}$, $\k_V[a_1,\ldots,a_n] = \k_k[a_{i_1},\ldots,a_{i_k}]$.  In this way, any free cumulant can be factored as a product of block cumulants $\k_k$.

\medskip

Equation \ref{free cumulants} is designed so that the following {\bf moment--cumulant formula} holds true:
\begin{equation}\label{moment--cumulant} \ff(a_1\cdots a_n) = \sum_{\pi\in NC(n)} \k_\pi[a_1,\ldots,a_n]. \end{equation}
Equations \ref{free cumulants} and \ref{moment--cumulant} show that there is a bijection between the mixed--moments and free cumulants of a collection of random variables $a_1,\ldots,a_n\in\mathscr{A}$.  The benefit of using the free cumulants in this context is their relation to freeness.  The following can be taken as the definition: {\em $a_1,\ldots,a_d\in\mathscr{A}$ are {\bf free} if and only if their mixed cumulants vanish.  That is, for any $n\in\N$ and collection $i_1,\ldots,i_n\in\{1,\ldots,d\}$ not all equal, $\k_n[a_{i_1},\ldots,a_{i_n}] = 0$.}
\medskip

Two important examples of operators with particularly nice free cumulants are circular operators and Haar unitaries.  If $c$ is circular, then among all free cumulants in $c$ and $c^\ast$, only $\k_2[c,c^\ast] = \k_2[c^\ast,c] = 1$ are non-zero.  On the other hand, for Haar unitary $u$, there are non-zero free cumulants of all even orders; the non-zero ones are
\[ \k_{2n}[u,u^\ast,\ldots,u,u^\ast] = \k_{2n}[u^\ast,u,\ldots,u^\ast,u] = (-1)^n C_{n-1}, \]
where $C_n$ is the Catalan number $\frac{1}{n+1}\binom{2n}{n}$.  In both cases ($c$ and $u$), the non-vanishing free cumulants must alternate between the operator and its adjoint.  This is the definition of $\mathscr{R}$-diagonality.

\begin{definition} \label{def R-diag} An operator $a$ in a $\mathrm{II}_1$-factor is called {\bf $\mathscr{R}$-diagonal} if the only non-zero free cumulants of $\{a,a^\ast\}$ are of the form $\k_{2n}[a,a^\ast,\ldots,a,a^\ast]$ or $\k_{2n}[a^\ast,a,\ldots,a^\ast,a]$. \end{definition}
The terminology ``$\mathscr{R}$-diagonal'' relates to the multi-dimensional $\mathscr{R}$-transform in \cite{Nica Speicher Book}; the joint $\mathscr{R}$-transform of $a,a^\ast$, for $\mathscr{R}$-diagonal $a$, has the form $(z,w)\mapsto \sum_{n\ge 0} \alpha_n (zw)^n + \beta_n (wz)^n$, and so is supported ``on the diagonal''.  In Section \ref{section 3}, we will use Definition \ref{def R-diag} directly.  Here, it is more convenient to have the following alternate characterization of $\mathscr{R}$-diagonality.
\begin{theorem*}[Theorem 15.10 in \cite{Nica Speicher Book}] An operator $a$ is $\mathscr{R}$-diagonal if and only if, for any Haar unitary $u$ $\ast$-free from $a$, $ua$ has the same distribution as $a$.
\end{theorem*}
\begin{remark} \label{remark polar decomp} To be clear, the above equi-distribution statement means that if $P$ is a non-commutative polynomial in two variables then $\ff(P(a,a^\ast)) = \ff(P(ua,a^\ast u^\ast))$.  Corollary 15.14 in \cite{Nica Speicher Book} re-interprets this equi-distribution property in terms of the polar decomposition of $a$: at least in the case that $\ker a = \{0\}$, $a$ is $\mathscr{R}$-diagonal iff its polar decomposition is of the form $a=ur$ where $r\ge 0$, $u$ is {\em Haar} unitary, and $r,u$ are $\ast$-free.  In the case that $a$ is $\mathscr{R}$-diagonal but $\ker a \ne \{0\}$, it is still possible to write $a = ur$ with $u$ Haar unitary and $r\ge 0$, but this is not the polar decomposition of $a$ and here $u,r$ are not $\ast$-free; cf.\ Proposition 15.13 in \cite{Nica Speicher Book}.
\end{remark}

Let $a$ be $\mathscr{R}$-diagonal.  For exponents $\e_1,\ldots,\e_n\in\{1,\ast\}$, our immediate aim is to appropriately bound general moments of the form $\ff(a^{\e_1}\cdots a^{\e_n})$.  Denote the string $(\e_1,\ldots,\e_n)$ as $\mx{S}$, and denote $a^{\e_1}\cdots a^{\e_n}$ by $a^{\mx{S}}$.  The following specialization of Equation \ref{moment--cumulant} is vital to the combinatorial understanding of $\mathscr{R}$-diagonal moments.

\begin{proposition} Let $\mx{S} = (\e_1,\ldots,\e_n)$ be a string of $1$s and $\ast$s. Let $NC(\mx{S})$ denote the set of all partitions $\pi\in NC(n)$ such that each block of $\pi$ is of even size and alternates between $1$ and $\ast$ in $\mx{S}$.  Let $a$ be $\mathscr{R}$-diagonal.  Then
\begin{equation} \label{R-diag moment--cumulant} \ff(a^\mx{S}) = \sum_{\pi\in NC(\mx{S})} \k_\pi[a^{\e_1},\ldots,a^{\e_n}]. \end{equation}
\end{proposition}
\noindent For example, consider the word $a^3 a^{\ast 2} a a^{\ast 2}$ with exponent string $\mx{S} = (1,1,1,\ast,\ast,1,\ast,\ast)$.  The set $NC(\mx{S})$ consists of the three partitions in Figure \ref{fig NC(S)}.
\begin{figure}[htbp]
\begin{center}
\input{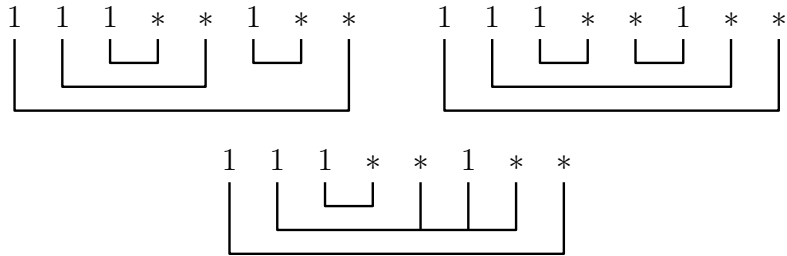}
\caption{The three partitions in $NC(1,1,1,\ast,\ast,1,\ast,\ast)$.}
\label{fig NC(S)}
\end{center}
\end{figure}

\begin{proof} In Equation \ref{moment--cumulant}, consider the general term $\k_\pi[a^{\e_1},\ldots,a^{\e_n}]$ in the summation.  This factors into terms $\k_V[a^{\e_1},\ldots,a^{\e_n}]$ over the blocks $V$ of $\pi$.  Since $a$ is $\mathscr{R}$-diagonal, the only such non-zero terms are of the form $\k_{2m}[a,a^\ast,\ldots,a,a^\ast]$ or $\k_{2m}[a^\ast,a,\ldots,a^\ast,a]$.  Thus, each $V$ must be even in size, and must alternate between $a$ and $a^\ast$ for the term to contribute.  It follows that the non-zero terms in the sum \ref{moment--cumulant} are all indexed by $\pi\in NC(\mx{S})$.
\end{proof}

\begin{corollary} \label{cor balanced} Let $a$ be $\mathscr{R}$-diagonal, and let $\mx{S}$ be a string.  Then $\ff(a^\mx{S})$ is $0$ unless $\mx{S}$ is {\em balanced}: it must have equal numbers of $1$s and $\ast$s.  In particular, $\mx{S}$ must have even length. \end{corollary}

\begin{proof} In each term $\k_\pi$ in \ref{R-diag moment--cumulant}, each block of the partition $\pi\in NC(\mx{S})$ is of even length and alternates between $1$ and $\ast$.  Hence, each block has equal numbers of $1$s and $\ast$s, and thus only balanced $\mx{S}$ contribute to the sum. \end{proof}

\begin{corollary} \label{cor rotation} $\mathscr{R}$-diagonal operators are rotationally-invariant: let $a$ be $\mathscr{R}$-diagonal, and let $\theta\in\R$.  Then $e^{i\theta}a$ has the same distribution as $a$. \end{corollary}

\begin{remark} If $a$ were a normal operator, then the above equi-distribution statement is precisely the same as requiring the spectral measure of $a$ to be a rotationally-invariant measure on $\C$. \end{remark}

\begin{proof} Let $P$ be a non-commutative monomial in two variables, $P(x,y) = x^{n_1}y^{m_1}\cdots x^{n_r}y^{m_r}$.  Set $n_1+\cdots+n_r = n$ and $m_1+\cdots +m_r =m$.  Then $P(e^{i\theta}a,e^{-i\theta}a^\ast) = e^{i(n-m)\theta} P(a,a^\ast)$.  By Corollary \ref{cor balanced}, if $n\ne m$ then $\ff(P(a,a^\ast))=\ff(e^{i(n-m)\theta}P(a,a^\ast)) =0$. Otherwise, the two elements are equal and so have the same trace. 
\end{proof}

The following orthogonality relation will be important in the sequel, and is an immediate consequence of Corollary \ref{cor balanced}.

\begin{corollary} \label{cor orthogonal} Let $a_1,a_2,\ldots,a_d$ be $\ast$-free $\mathscr{R}$-diagonal operators in $(\mathscr{A},\ff)$.  Then $(a_j)^n$ and $(a_k)^m$ are orthogonal in $L^2(\mathscr{A},\ff)$ whenever $j\ne k$ or $n\ne m$. \end{corollary}

\begin{proof} The inner product is $\ff(a_j^n (a_k^m)^\ast)$.  Applying Equation \ref{moment--cumulant}, this is a sum of terms of the form $\k_\pi[a_j,\ldots,a_j,a_k^\ast,\ldots,a_k^\ast]$.  If $j\ne k$, this is a mixed cumulant of $\ast$-free random variables, and so vanishes.  If $j=k$ and $m\ne n$, the inner product is $\phi(a^\mx{S})$ for an imbalanced string, and so it also vanishes by Corollary \ref{cor balanced}. \end{proof}

\subsection{Completely positive extensions for rotationally-invariant generators} Here we prove Theorem \ref{theorem 1}, which actually holds in the wider context of $\ast$-free rotationally-invariant generators.

\begin{proof}[Proof of Theorem \ref{theorem 1}] Let $a_1,\ldots,a_d$ be $\ast$-free rotationally-invariant operators.  Since the law of the generators determines the von Neumann algebra they generate, for any $\theta\in\R$ and any index $j\in\{1,\ldots,d\}$ there is a $\ast$-automorphism
\[ \alpha^{(j)}_\theta\colon W^\ast(a_j)\to W^\ast(a_j) \]
determined by $\alpha^{(j)}_\theta(a_j) = e^{i\theta} a_j$.  Since the generators are $\ast$-free, $W^\ast(a_1,\ldots,a_d)$ is naturally isomorphic to $W^\ast(a_1)\ast \cdots \ast W^\ast(a_d)$, and so we have a $\ast$-automorphism
\begin{equation}\label{automorphism} \alpha_\theta = \alpha^{(1)}_\theta\ast \cdots \ast \alpha^{(d)}_\theta \colon W^\ast(a_1,\ldots,a_d)\to W^\ast(a_1,\ldots,a_d). \end{equation}
Note, any $\ast$-automorphism is automatically completely-positive.

\medskip

Let $P(r,\theta)$ denote the Poisson kernel for the unit disc in $\C$; that is, for $r\in[0,1]$ and $\theta\in[0,2\pi)$,
\begin{equation} \label{Poisson kernel} P(r,\theta) = \mathrm{Re}\,\frac{1+re^{i\theta}}{1-re^{i\theta}} = \frac{1-r^2}{1-2r\cos\theta +r^2} = \sum_{k=-\infty}^\infty r^{|k|}e^{ik\theta}. \end{equation}
For any fixed $r<1$, this kernel is strictly positive and bounded on the circle $\theta\in[0,2\pi)$.  For any operator $x\in W^\ast(a_1,\ldots,a_d)$, define
\[ D^r x = \frac{1}{2\pi}\int_0^{2\pi} P(r,\theta)\,\alpha_\theta(x)\,d\theta. \]
Because the Poisson kernel is continuous and bounded on a compact set, this integral converges in operator norm.  Each uniformly convergent Riemann sum is therefore of the form $\sum p_j \alpha_{\theta_j}$, where the $p_j$ (samples of the Poisson kernel) are positive numbers.  Each such sum is thus completely positive, and so the uniform limit $D^r$ is a completely positive operator as well.

\medskip

Now we need only check the action of $D^r$ on monomials $a_{i_1}^{\e_1}\cdots a_{i_n}^{\e_n}$, where $\e_j\in\{1,\ast\}$.  Note that
\[ \alpha^{(i_k)}_\theta (a_{i_1}^{\e_1}\cdots a_{i_k}^{\e_k}\cdots a_{i_n}^{\e_n}) = e^{\e_ki\theta} \,(a_{i_1}^{\e_1}\cdots a_{i_k}^{\e_k}\cdots a_{i_n}^{\e_n}), \]
where $\e=\ast$ is interpreted as $\e=-1$ on the right-hand-side.  Hence,
\[ \alpha_\theta(a_{i_1}^{\e_1}\cdots a_{i_n}^{\e_n}) = e^{i(\e_1+\cdots+\e_n)\theta} (a_{i_1}^{\e_1}\cdots a_{i_n}^{\e_n}). \]
Hence, from the third equality in Equation \ref{Poisson kernel},
\[ P(r,\theta)\,\alpha_\theta(a_{i_1}^{\e_1}\cdots a_{i_n}^{\e_n}) = \sum_{k=-\infty}^\infty r^{|k|} e^{i(k+\e_1+\cdots +\e_n)\theta} (a_{i_1}^{\e_1}\cdots a_{i_n}^{\e_n}). \]
Integrating term-by-term around the circle, the only term that survives is $k = -(\e_1+\cdots+\e_n)$, and the integral there is just $1$.  Hence,
\[ D^r(a_{i_1}^{\e_1}\cdots a_{i_n}^{\e_n}) = r^{|\e_1+\cdots+\e_n|}\, a_{i_1}^{\e_1}\cdots a_{i_n}^{\e_n}. \]
Setting $D_t = D^{e^{-t}}$ yields the formula in Theorem \ref{theorem 1}.  Note that this restricts to the dilation semigroup when all $\e_j=1$.  Hence, $D_t$ has a completely-positive extension.
\end{proof}

Indeed, every dilation semigroup associated to $\ast$-free rotationally-invariant generators (for example $\ast$-free $\mathscr{R}$-diagonal generators) has a completely positive extension.  In the case of a single unitary generator $u$, the action of $D_t$ on Laurent polynomial in $u$ is simply $D_t (u^n) = e^{-|n|t}\,u^n$, which ``counts unitaries''.  However, this action is not particularly natural in the general setting: it has no connection with the distribution of the generators.  In the circular setting, this is {\em not} the free O--U semigroup, which is also completely positive (as proved in \cite{Biane 2}).  Therefore, this extension may not be unique.

\medskip

\subsection{Completely positive extensions for Markov kernels}  A different CP extension is possible for {\em some} generating distributions.  Let $\mu$ be a (compactly-supported) probability measure on $\R$.  Then the monomials $\{1,x,x^2,\ldots\}$ are dense in $L^2(\R,\mu)$, and Gram-Schmidt orthogonalization produces the {\em orthogonal polynomials} $\{p_0,p_1,p_2,\ldots\}$ associated to $\mu$.  (If $\mu$ has infinite support, all monomials are linearly independent; if $\mu$ has support of size $n$ then $p_k = 0$ for $k>n$.)  The polynomial $p_n$ has degree $n$, and $p_0(x)=1$ while $p_1(x)=x$.  If $\mu$ is the semicircle law, the associated polynomials are the Tchebyshev polynomials of type II, usually denoted $u_n$.

\medskip

Given $\mu$ with associated orthogonal polynomials $\{p_n\}$, let $\hat{p}_n$ denote the normalized polynomials (in $L^2(\mu)$, so that $\{\hat{p}_n\}$ forms an o.n.\ basis).  Consider the following integral kernel (which may take infinite values):
\begin{equation} \label{kernel} m_\mu(r;x,y) = \sum_{n\ge 0} r^n\, \hat{p}_n(x)\hat{p}_n(y). \end{equation}
Here $r\in[0,1)$ and $x,y$ range over $\supp\mu$.  The formula converges at least when $r|xy|<1$.  We refer to the kernel in Equation \ref{kernel} as a {\em Mehler kernel}.  If the measure is chosen as the standard normal law (which is not compactly-supported but has sufficient tail decay to ensure the $L^2$-density of polynomials), then the polynomials $p_n$ are the Hermite polynomials and the kernel is {\em the} Mehler kernel.  On the other hand, if $\mu$ is the semicircle law, setting $r=e^{-t}$ yields the kernel of the (one-dimensional) free O--U semigroup in \cite{Biane 2}.

\medskip

Let $M_\mu(r)$ denote the integral operator associated to the kernel $m_\mu(r,\cdot,\cdot)$; that is, for $f\in L^\infty(\mu)$ at least, let
\[ (M_\mu(r) f)(x) = \int_{\R} m_\mu(r;x,y)\,f(y)\,d\mu(y). \]
Since $\mu$ is a probability measure, $L^\infty(\mu)\subset L^2(\mu)$ and so any such $f$ has an $L^2$-expansion $f = \sum_{n\ge 0} f_n\,\hat{p}_n$ for an $\ell^2(\N)$--sequence $(f_n)_{n=0}^\infty$.   From the orthonormality of the polynomials $\hat{p}_n$ in $L^2(\mu)$ it is then easy to see that the action of $M_\mu(r)$ is
\begin{equation} \label{Mehler action} M_\mu(r) f = \sum_{n\ge 0} r^n\,f_n\,\hat{p}_n. \end{equation}
That is, $M_\mu(r)$ is a polynomial multiplier semigroup (in multiplicative form): $M_\mu(r)\,p_n = r^n\,p_n$.  Contingent on convergence in $L^\infty(\mu)$, we may then ask the question of whether $M_\mu(r)$ is completely positive.  In this case, as a bounded operator on the commutative von Neumann algebra $L^\infty(\mu)$, complete positivity is equivalent to positivity of the kernel: $M_\mu(r)$ is CP if and only if $m_\mu(r;x,y) \ge 0$ for $x,y\in\supp\mu$.  In other words, $M_\mu$ is CP if and only if $m_\mu$ is a {\em Markov kernel}.

\begin{example} \label{example Bernoulli} For the point mass $\delta_0$, $p_0 = 1$ and all other $p_n$ are $0$, so the kernel $m_{\delta_0}$ is trivially Markovian.  Let $\l>0$ and set $\mu = \frac{1}{2}(\delta_\l + \delta_{-\l})$.  Then we may easily calculate that $\hat{p}_0(x) = 1$ and $\hat{p}_1(x)=x/\l$, while all higher polynomials are $0$.  Thence $m_\mu(r;x,y) = 1+rxy/\l^2$, and on the support of $\mu$ $x,y\in\{\pm\l\}$ we have $m_\mu = 1+r \ge 0$.  Hence $m_\mu$ is Markovian. \end{example}

\begin{example} On the other hand, consider an arbitrary symmetric measure with $3$-point support, $\nu = a(\delta_{\l}+\delta_{-\l})+(1-2a)\delta_0$ where $0\le a\le\frac{1}{2}$ and $\l>0$. A simple calculation shows that in this case
\[ m_\nu(r;\l x,\l y) = 1 + r\frac{xy}{2a} + r^2\frac{(x^2-2a)(y^2-2a)}{2a(1-2a)}. \]
The arguments $\l x,\l y$ are in $\supp\nu$ if and only if $x,y\in\{0,\pm 1\}$; note that
\[ m_\nu(r;1,-1) = 1-r\frac{1}{2a} + r^2\frac{1-2a}{2a} = (1-r)\left[1+r-\frac{r}{2a}\right]. \]
If $a<\frac{1}{4}$, this is $<0$ for some $r\in[0,1)$, and so $m_\nu$ is {\em not} Markovian for some choices of $a$.
\end{example}

It is a historically challenging problem to determine, for a given measure $\mu$, whether the associated Mehler kernel $m_\mu$ is a Markov kernel.  (See, for example, \cite{Wlodek 1,Wlodek 2,Wlodek 3,Wlodek 4}.)  The motivating example ({\em the} Mehler kernel) is Markovian: with $\mu = \gamma_1$ (Gauss measure on $\R$),
\[ m_{\gamma_1}(r;x,y) = (1-r^2)^{-1/2}\,\exp\left(y^2/2 +(1-r^2)^{-1/2}(rx-y)\right), \]
which is strictly positive on $\R = \supp\gamma_1$ for $0\le r<1$.  Writing a formula for a general Mehler kernel is a hopeless task.  Nevertheless, the following positivity condition affords many examples of Markovian Mehler kernels.

\begin{proposition} \label{positivity} The following conditions are equivalent.
\begin{enumerate}
\item \label{1} $m_\mu(r;x,y)\ge 0$ a.s.\ on $\mathrm{supp}\,\mu$.
\item \label{2} For $f\ge 0$ a.s.\ on $\mathrm{supp}\,\mu$, and in $L^1(\mu)$, $M_\mu(r)f \ge 0$ a.s.\ on $\mathrm{supp}\,\mu$.
\item \label{3} For $f\in L^1(\mu)$, $\| M_\mu(r) f \|_{L^1(\mu)} \le \|f\|_{L^1(\mu)}$.
\item \label{4} For $f\in L^\infty(\mu)$, $\| M_\mu(r) f \|_{L^\infty(\mu)} \le \|f\|_{L^\infty(\mu)}$.
\item \label{5} For {\em all} $p\in [1,\infty]$, $\| M_\mu(r) f \|_{L^p(\mu)} \le \|f\|_{L^p(\mu)}$.
\end{enumerate}
\end{proposition}

\begin{remark} The statement is that $m_\mu$ is Markovian if and only if $M_\mu$ is a contraction on $L^1$, or on $L^\infty$, or on $L^p$ for all $p$ between $1$ and $\infty$.  This proposition and the following proof are borrowed from \cite{Janson}, but the results really go back to Beurling and Deny \cite{BD}.  \end{remark}

\begin{proof} The equivalence of (\ref{1}) and (\ref{2}) is elementary.  Condition (\ref{4}) follows from (\ref{3}) by duality, since $M_\mu(r)$ is self-adjoint on $L^2(\mu)$. Condition (\ref{5}) follows from (\ref{3}) and (\ref{4}) by the Riesz-Thorin interpolation theorem, and evidently (\ref{5}) implies (\ref{3}). 

\medskip

Suppose condition (\ref{1}) holds.  Then
\[ \begin{aligned} |M_\mu(r) f (x)| = \big| \int m_\mu(r;x,y) f(y)\,d\mu(y)\big| &\le \int |m_\mu(r;x,y)|\,|f(y)|\,d\mu(y) \\
&= \int m_\mu(r;x,y)\,|f(y)|\,d\mu(y) = M_\mu(r)|f|(x). \end{aligned} \]
The reader can easily check that $M_\mu(r)$ is trace-preserving: $\int M_\mu(r)g\,d\mu = \int g\,d\mu$.  Hence $\int M_\mu(r)|f|\,d\mu = \int |f|\,d\mu = \|f\|_{L^1(\mu)}$, and so
\[ \|M_\mu(r) f\|_{L^1(\mu)} = \int |M_\mu(r)f|\,d\mu \le \int M_\mu(r)|f|\,d\mu = \|f\|_{L^1(\mu)}, \]
verifying property (\ref{3}).

\medskip

On the other hand, suppose condition (\ref{3}) holds.  Take $f\in L^1(\mu)$ with $f\ge 0$.  Then by assumption (\ref{3}),
\[ \int |M_\mu(r)f(x)|\,d\mu(x) = \|M_\mu(r) f\|_{L^1(\mu)} \le \|f\|_{L^1(\mu)} = \int f\,d\mu, \]
and as above we have $\int f\,d\mu = \int M_\mu(r) f\,d\mu$.  Hence, $\int |M_\mu(r)f|\,d\mu \le \int M_\mu(r)f\,d\mu$.  Since $\mu$ is a positive measure, this means that $M_\mu(r) f(x)\ge 0$ for $x\in\mathrm{supp}\,\mu$, verifying property (\ref{2}).
\end{proof}

\begin{example} In \cite{BoKS} the authors introduced the {\bf $q$-Gaussian factors}, with their associated $q$-Gaussian measures $\sigma_q$.  When $q=1$, $\sigma_q = \gamma_1$ is the standard normal law; when $q=0$, $\sigma_0$ is the semicircle law, and $q=-1$ yields a two-point Bernoulli measure as in Example \ref{example Bernoulli}.  All the measures $\sigma_q$ with $-1\le q<1$ are compactly supported and symmetric.  The associated orthogonal polynomials are the {\bf $q$-Hermite polynomials $H^{(q)}_n$} given by the following tri-diagonal recursion:
\[ H^{(q)}_0(x) = 1, H^{(q)}_1(x) = x, \quad H^{(q)}_{n+1}(x) = x\,H^{(q)}_n(x) + [n]_q\,H^{(q)}_{n-1}(x), \]
where $[n]_q = 1+q+\cdots + q^{n-1}$.  When $q=1$ these are the Hermite polynomials, associated to Gauss measure; for $q=0$ the recurrence produces the Tchebyshev II polynomials, orthogonal for the semicircle law.  The $L^2(\sigma_q)$ normalization factor is $1/[n]_q!$ where $[n]_q! = [n]_q\cdot [n-1]_q\cdots [2]_q\cdot [1]_q$.

\medskip

\noindent The Mehler kernel $m_{\sigma_q}$ is the kernel of the $q$-O--U semigroup considered in \cite{Biane 2}.  There, Biane proved Nelson's hypercontractivity inequalities for the associated semigroup, which include as a special case condition \ref{5} in Proposition \ref{positivity}.  Hence, there is a continuous family of Mehler kernels that are both symmetric and Markovian.

\end{example}

We now come to the question of completely positive extensions for $\mathscr{R}$-diagonal dilation semigroups.  Let $a_1,\ldots,a_d$ be $\ast$-free $\mathscr{R}$-diagonal operators in a $\mathrm{II}_1$-factor (traciality is necessary here).  From \cite{Nica Speicher Duke Paper}, there are self-adjoint even elements $x_j$ (that is, the distribution $\mu_{x_j}$ is symmetric on $\R$) with the same free cumulants as those of $a_j$.  Hence,
\begin{equation} \label{even} \k_{2n}[x_j,x_j,\ldots,x_j,x_j] = \k_{2n}[a_j,a_j^\ast,\ldots,a_j,a_j^\ast] = \k_{2n}[a_j^\ast,a_j,\ldots,a_j^\ast,a_j]. \end{equation}
(The odd cumulants of $a_j,a_j^\ast$ are $0$ by definition, and so are those of $x_j$ since it is even; all odd moments are $0$, and so too are all odd cumulants.)  This means that $|a_j|$ is equal in distribution to $|x_j|$.  In \cite{Haagerup Larsen} (Corollary 3.2), the authors show that if $s$ is self-adjoint, even, free from $x_j$ and $s^2=1$, then $sx_j$ is $\mathscr{R}$-diagonal and indeed has the same distribution as $a_j$.  What's more, the construction of $x_j$ from $a_j$ takes place within the $W^\ast$-algebra generated by $a_j$, and so $x_1,\ldots,x_d$ are free.  In other words, we may represent the generators $a_j$ in the form $a_1 = sx_1,\ldots,a_d = sx_d$ where $x_1,\ldots,x_d,s$ are all free, self-adjoint, even, and $s^2=1$.  Note, then, that
\[ W^\ast(a_1,\ldots,a_d) \cong W^\ast(x_1,\ldots,x_d,s). \]
Let $\mu_j$ denote the distribution of $x_j$ on $\R$.  Since the $x_j$ are free, we have
\[ W^\ast(a_1,\ldots,a_d) \cong L^\infty(\mu_1)\ast \cdots \ast L^\infty(\mu_d)\ast W^\ast(s). \]
We may then define, for $0\le r<1$, the operator $T^r$ on $L^\infty(\mu_1)\ast \cdots \ast L^\infty(\mu_d)\ast W^\ast(s)$ by 
\begin{equation} \label{Markov extension} T^r = M_{\mu_1}(r)\ast\cdots\ast M_{\mu_d}(r)\ast \mathrm{Id}. \end{equation}
Now, suppose that the kernels $m_{\mu_1},\ldots,m_{\mu_d}$ are in fact {\em Markovian} (for example, satisfying the conditions of Proposition \ref{positivity}).  Then the operators $M_{\mu_1}(r),\ldots,M_{\mu_d}(r)$ are all completely positive (they are positive operators on commutative von Neumann algebras).  Moreover, as was stated in the proof of Proposition \ref{positivity}, they are trace-preserving ($\int M_\mu f\,d\mu = \int f\,d\mu$).   Of course the $\mathrm{Id}$ map on $W^\ast(s)$ is also CP and trace preserving.  Then by Theorem 3.8 in \cite{Blanchard Dykema}, the operator $T^r$ is completely positive and trace preserving on $L^\infty(\mu_1)\ast\cdots\ast L^\infty(\mu_d)\ast W^\ast(s) \cong W^\ast(a_1,\ldots,a_d)$. 

\medskip

Now, by definition the orthogonal polynomial $p^{\mu}_1(x)$ for any measure $\mu$ is a scalar multiple of $x$, and hence $M_{\mu_j}(r)\,x_j = r\,x_j$.  Then the action of $T^r$ on words in the generators and not their adjoints is:
\[ \begin{aligned} T^r(a_{i_1}a_{i_2}\cdots a_{i_n}) &= T^r(sx_{i_1}sx_{i_2}\cdots sx_{i_n}) \\
&= s(r\,x_{i_1})s(r\,x_{i_2})\cdots s(r\,x_{i_n}) = r^n a_{i_1}a_{i_2}\cdots a_{i_n}. \end{aligned} \]
Setting $T_t = D^{e^{-t}}$, this means that $T_t$, restricted to the (non--$\ast$) algebra generated by $a_1,\ldots,a_d$, is the associated $\mathscr{R}$-diagonal dilation semigroup $D_t$.  We have therefore proved the following.

\begin{theorem} \label{theorem 2} Let $a_1,\ldots,a_d$ be $\ast$-free $\mathscr{R}$-diagonal operators, and suppose that the Mehler kernels associated to the symmetrizations $\mu_j$ of the distributions of $|a_j|$ are Markovian.  Then Equation \ref{Markov extension} with $T_t = D^{e^{-t}}$ defines a CP trace preserving extension of the $\mathscr{R}$-diagonal dilation semigroup $D_t$ of $\{a_1,\ldots,a_d\}$ which is different from the extension of Theorem \ref{theorem 1}.  In particular, in the case that each $a_j$ is circular, $T_t$ corresponds to the free O--U semigroup.
\end{theorem}

\begin{remark} The notion of correspondence in the final statement of Theorem \ref{theorem 2} is as follows.  The free O--U semigroup $U_t$ acts on $W^\ast(\s_1,\ldots,\s_d)$ where $\s_j = (c_j+c_j^\ast)/\sqrt{2}$ are free semicircular operators -- the $c_j$ are $\ast$-free circular operators.  In this circular case, the symmetrization of the distribution of $|c_j|$ is also the semicircle law $\s$: this is easy to check from Equation \ref{even} and the fact that $\k_2[c,c^\ast]=\k_2[c^\ast,c] = \k_2[\s,\s] =1$ and all other free cumulants are $0$.  Hence, the construction above $W^\ast(c_1,\ldots,c_d) \cong W^\ast(x_1,\ldots,x_d,s)$ yields free semicircular $x_j$, and so we can view $U_t$ acting in $W^\ast(x_1,\ldots,x_d)$.  Its action is
\[ U_t\left(u_{n_1}(x_{i_1}) \cdots u_{n_k}(x_{i_k})\right) = e^{-(n_1+\cdots+n_k)t}u_{n_1}(x_{i_1}) \cdots u_{n_k}(x_{i_k}), \]
where $u_n$ are the Tchebyshev II polynomials, the orthogonal polynomials for the semicircle law, and the indices $i_\ell$ are consecutively distinct.  From Equation \ref{Mehler action}, this is precisely the action of $M_\s(e^{-t})\ast\cdots\ast M_\s(e^{-t})$ on $W^\ast(x_1,\ldots,x_d) \cong L^\infty(\s)\ast\cdots\ast L^\infty(\s)$, and so $U_t$ is the restriction from $W^\ast(x_1,\ldots,x_d,s)$ to $W^\ast(x_1,\ldots,x_d)$ of the Markov extension $T_t$ of the circular $\mathscr{R}$-diagonal dilation semigroup.  Note that $T_t$ acts on the full von Neumann algebra $W^\ast(x_1,\ldots,x_d,s)$ through a very similar formula:
\[ \begin{aligned} T_t\big(s^{\e_0} u_{n_1}(x_{i_1}) s^{\e_1}  u_{n_2}(x_{i_2})  \cdots& s^{\e_{n-1}}  u_{n_k}(x_{i_k}) s^{\e_n}\big) \\
&= e^{-(n_1+\cdots+n_k)t}\,s^{\e_0} u_{n_1}(x_{i_1}) s^{\e_1}  u_{n_2}(x_{i_2})  \cdots s^{\e_{k-1}}  u_{n_k}(x_{i_k}) s^{\e_k}, \end{aligned} \]
where $\e_j$ is either $0$ or $1$ (i.e.\ either the $s$ is included or not). In this case, if an $s$ separates $u_n(x_i)$ from $u_m(x_j)$, it is not required that $i\ne j$; the $s$ stands in for free product.
\end{remark}

\begin{remark} \label{remark Tt expansion} The orthogonal polynomials $\{p_0,p_1,p_2,\ldots\}$ for $\mu$ are constructed from $\{1,x,x^2,\ldots\}$ by Gram-Schmidt orthogonalization, which means that the span (in $L^\infty$) of $\{p_0,\ldots,p_n\}$ is the same as the span of $\{1,x,\ldots,x^n\}$.  As a result, and monomial $x^n$ can be expanded $x^n = \sum_{k=0}^n \alpha_{n,k}\,p_k(x)$ as a finite sum.  Since $\mu$ is symmetric, $p_n$ is even if $n$ is even and odd if $n$ is odd, and so it follows that only every second $\alpha_{n,k}$ is non-zero.  What's more, the leading term of $p_n$ is $x^n$, and so $\alpha_{n,n}=1$.  So, for example,  $x^3 = p_3(x)+\alpha_{3,1}x$.   (The other coefficients are certain combinations of the moments of the measure $\mu$; for example, $\alpha_{3,1} = \int x^4\,d\mu/\int x^2\,d\mu$.)  This allows for the easy determination of the action of $T_t$ on arbitrary words in the generators $a_1,\ldots,a_d$ and their adjoints. \end{remark}

\begin{example}  Consider the word $a_1^\ast a_1a_1^\ast a_2^2 a_1^\ast$.  We rewrite this as
\begin{equation} \label{a to x} \begin{aligned} a_1^\ast a_1a_1^\ast a_2^2 a_1^\ast &= (x_1s)(sx_1)(x_1s)(sx_2)^2(x_1s)   \\
&= x_1^3\,x_2\,s\,x_2\,x_1s.
\end{aligned} \end{equation}
Let $p_n$ denote the orthogonal polynomials of the distribution of $x_1$, and let $\alpha_{n,k}$ be the relevant coefficients, as explained in Remark \ref{remark Tt expansion}.  Then $x_1^3 =p_3(x_1)+ \alpha_{3,1}x_1$, and so
\[ \begin{aligned} T_t(a_1^\ast a_1a_1^\ast a_2^2 a_1^\ast) &= \left(e^{-3t}\,p_3(x_1)+e^{-t}\alpha_{3,1}x_1\right)(e^{-t}x_2)s(e^{-t}x_2)(e^{-t}x_1)s \\
&=  e^{-6t}\,p_3(x_1)x_2\,s\,x_2\,x_1s+e^{-4t}\alpha_{3,1}\,x_1\,x_2\,s\,x_2\,x_1s. \end{aligned} \]
Rewriting $p_3$ as $p_3(x_1) = x_1^3 - \alpha_{3,1}x_1$ yields
\[  T_t(a_1^\ast a_1a_1^\ast a_2^2 a_1^\ast) = 
  e^{-6t}x_1^3\,x_2\,s\,x_2\,x_1s+(e^{-4t}-e^{-6t})\alpha_{3,1}\,x_1\,x_2\,s\,x_2\,x_1s. \]
From Equation \ref{a to x} we have $x_1^3\,x_2\,s\,x_2\,x_1s = a_1^\ast a_1a_1^\ast a_2^2 a_1^\ast$ ; for the second term, we introduce $1=s^2$ between $x_1\,x_2$ yielding,
\[ x_1\,x_2\,s\,x_2\,x_1s = (x_1s)(sx_2)(sx_2)(x_1s) = a_1^\ast a_2^2 a_1^\ast. \]
In this way, any monomial in $x_1,\ldots,x_d,s$ may be converted into a unique monomial in $a_1,,\ldots,a_d$ and their adjoints.  Hence, for this example, we have
\begin{equation} \label{example Tt} T_t(a_1^\ast a_1a_1^\ast a_2^2 a_1^\ast)
=  e^{-6t}a_1^\ast a_1a_1^\ast a_2^2 a_1^\ast + (e^{-4t}-e^{-6t})\alpha_{3,1}\,a_1^\ast a_2^2 a_1^\ast. \end{equation}
\end{example}

\begin{remark} Equation \ref{example Tt} is typical of the action of the Markov extension $T_t$ of $D_t$ (when it exists).  The leading term is multiplication by $e^{-nt}$ on any word of length $n$, mimicking the action of $D_t$ on words in the generators and not their inverses; this is to be contrasted with the alternating degree count in the generic extension of Theorem \ref{theorem 1}.  There are then correction terms involving lower-degree words, with coefficients that are polynomial in $e^{-t}$ that vanish at $t=0$.  All such correction terms vanish in the case of words in the generators without inverses, thus resulting in the $\mathscr{R}$-diagonal dilation semigroup as per Theorem \ref{theorem 2}.
\end{remark}

\section{Optimal Ultracontractivity}\label{section 3}

We now wish to consider the action of an $\mathscr{R}$-diagonal dilation semigroup $D_t$, relative to generators $a_1,\ldots,a_d$, on $L^2_{hol}(a_1,\ldots,a_d)$ taking values in $L^p(W^\ast(a_1,\ldots,a_d),\ff)$ for $p=2,4,\ldots,\infty$ (where the $L^\infty$ is just $W^\ast(a_1,\ldots,a_d)$ equipped with its operator norm).  The initial idea is to approximate the operator norm through the $L^p$--norms as $p\to\infty$ to prove Theorem \ref{theorem 3}.  In fact, the following approximations are slightly too weak to make this approach work; the resultant lower-bound involves a constant which tends to $0$ as $p\to\infty$, missing the target at $p=\infty$ by an infinitesimal exponent.  A separate argument based on similar combinatorics is given to prove Theorem \ref{theorem 3}.  The $L^p$-estimates are included below for independent interest; in particular, they are {\em in line} with the conjecture that the spaces $L^p_{hol}(a_1,\ldots,a_d)$ are complex interpolation scale.

\subsection{Bounding $|NC(\mx{S})|$} A generic string may be written in the form $\mx{S} = (1^{n_1},\ast^{m_1},\ldots,1^{n_r},\ast^{m_r})$ where $r\ge 1$ and $n_1,m_1,\ldots,n_r,m_r \ge 1$.  (There are implied commas: $1^3 = (1,1,1)$.)  The number $r$ of alternations between $1$ and $\ast$ is an important statistic; refer to it as the number of {\bf runs} in $\mx{S}$.  We will shortly provide a lower-bound on the size of $NC(\mx{S})$, which depends fundamentally on this number $r$.  First, we specialize $NC(\mx{S})$ to pairings.

\begin{definition} \label{def NC2(S)} Given a string $\mx{S}$, let $NC_2(\mx{S})$ denote the set of all pairings in $NC(\mx{S})$. \end{definition}
\noindent For example, for the string in Figure \ref{fig NC(S)}, the first two partitions are in $NC_2(\mx{S})$ (the last is not).  Of course $NC_2(\mx{S})$ is a subset of $NC(\mx{S})$, so we may estimate its size to find a lower-bound on $|NC(\mx{S})|$.  It follows from Equation \ref{R-diag moment--cumulant} and the form of the cumulants of a circular operator $c$ that $\ff(c^\mx{S}) = |NC_2(\mx{S})|$ for any $\mx{S}$, and so it is no surprise that this set plays an important role in the following estimates. For the regular strings
$\mx{S}=(1^n,\ast^n)^r$,  $|NC_2(\mx{S})|$ was calculated exactly in our paper \cite{Kemp Speicher}.  Our proof was very topological, but a simpler recursive proof is given in \cite{REU}.  The result is as follows:
\begin{equation}\label{Fuss-Catalan} |NC_2\left((1^n,\ast^n)^r\right)| = \phi\left((c^nc^{\ast n})^r\right) = C^{(n)}_r =  \frac{1}{nr+1}\binom{(n+1)r}{r}. \end{equation}
The numbers $C^{(n)}_r$ are called {\em Fuss-Catalan numbers}.  As a function of $n$, $C^{(n)}_r$ is on the order of $(n+1)^{r-1}$.  This structure is reflected in the following estimate.

\begin{proposition} \label{prop NC2 lower bound} Let $\mx{S} = (1^{n_1},\ast^{m_1},\ldots,1^{n_r},\ast^{m_r})$ be a balanced string, and let $i$ be the minimum block size, $i=\min\{n_1,m_1,\ldots,n_r,m_r\} \ge 1$.  Then
\[ |NC_2(\mx{S})| \ge (1+i)^{r-1}. \]
\end{proposition}

\begin{proof} The case $r=1$ is simple: this means that $\mx{S}$ has the form $(1^{n_1},\ast^{m_1})$, and since $\mx{S}$ is balanced, this means $n_1=m_1=i$. It is therefore a regular string of the form mentioned above, and so $|NC_2(\mx{S})| = C^{(i)}_1 = \frac{1}{n_1+1}\binom{(i+1)1}{1} = 1 = (1+i)^{1-1}$, proving this base case correct.

\medskip

Proceeding by induction on $r\ge 2$, suppose that for any string $\widetilde{\mx{S}}$ with precisely $r-1$ runs, it holds that
$|NC_2(\widetilde{\mx{S}})| \ge (1+\tilde{i})^{r-2}$, where $\tilde{i}$ is the minimal block size in $\widetilde{\mx{S}}$.  Let $\mx{S}=(1^{n_1},\ast^{m_1},\ldots,1^{n_r},\ast^{m_r})$ be any balanced string with $r$ runs, and minimum block size $i$.  We may cyclically permute the entire string without affecting the size of $|NC_2(\mx{S})|$, and so without loss of generality we may assume that $n_1=i$.  (Note: if the minimum occurs on a $\ast$ block, we can rotate and then reverse the roles of $1$ and $\ast$.)  Now, since $n_1=i$ is the minimum, $m_r \ge i$, and so one possible way to pair each of the initial $1$s in $\mx{S}$ is with the last $i$ $\ast$s in the final block; the resulting leftover string $\widetilde{\mx{S}}$ is $(\ast^{m_1},1^{n_2},\ldots,\ast^{m_{r-1}},1^{n_r},0^{m_r-i})$, which can be rotated to the string $(1^{n_2},\ast^{m_2},\ldots,1^{n_r},\ast^{m_1+m_r-i})$ which is, by construction, still balanced, and has $r-1$ runs.  Therefore, by the induction hypothesis, there are at least $(1+\tilde{i})^{r-2}$ pairings of this internal string, and since it is a substring of $\mx{S}$, $\tilde{i}\ge i$; hence, with the initial block of $1$s all paired at the end, there are at least $(1+i)^{r-2}$ pairings in $NC_2(\mx{S})$.

\medskip

More generally, let $1\le \ell \le i = n_1$.  Since $m_1 \ge i \ge \ell$, the last $\ell$ $1$s in this first block can be paired to the first
$\ell$ $\ast$s, with the remaining $i-\ell$ $1$s pairing to the final $i-\ell \le m_r$ $\ast$s in the final block, as above.  The remaining internal string is then $1^{m_1-\ell},1^{n_2},\ldots,\ast^{m_{r-1}},1^{n_r},\ast^{m_r-(i-\ell)}$ which can be rotated to $1^{n_2},\ast^{m_2},\ldots,1^{n_r},\ast^{m_1+m_r-i}$ once again.
\begin{figure}[htbp]
\begin{center}
\input{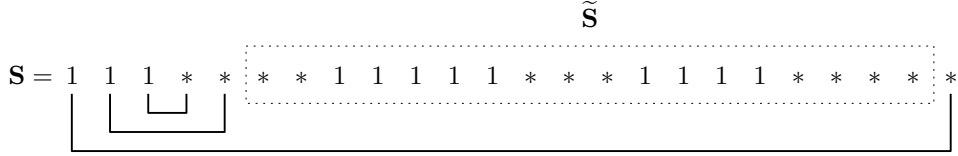}
\caption{One of the $i+1$ configurations for the first block of $1$s, yielding all the pairings of $\tilde{\mx{S}}$; in this example, $i=3$, and $\ell=2$.}
\label{fig NC_2 lower bound proof}
\end{center}
\end{figure}
Thus, as above, for each choice of $\ell$ between $1$ and $i$, we have at least $(1+i)^{r-2}$ distinct pairings of $\mx{S}$, and the different pairings for different $\ell$ are distinct.  Adding these $i(1+i)^{r-2}$ pairings to the $(1+i)^{r-2}$ in the case above, we see that $NC_2(\mx{S})$ indeed contains at least $(1+i)^{r-1}$ pairings.
\end{proof}

\begin{remark} This proof actually yields a somewhat larger lower-bound, given as a product of iterated minima $(1+i_1)(1+i_2)\cdots(1+i_{r-1})$ where $i_1 = i$ is the global minimum and each $i_{k+1}$ is the minimum of the leftover string after the inductive step has been applied at stage $k$ (i.e.\ $i_2 = \tilde{i}$ from the proof).  It is possible to construct examples where this iterated minimum product is much larger than the stated lower bound; it is also easy to construct strings with arbitrary length that achieve the bound. Regardless, the result of Proposition \ref{prop NC2 lower bound} is sufficient for our purposes. \end{remark}

\medskip

We will also require an upper-bound for the size of $|NC_2(\mx{S})|$. To achieve it, we need a convenient way to understand the restrictions a string $\mx{S}$ enforces over pairings.  Consider the string $(1,1,1,1,\ast,\ast,1,1,\ast,\ast,\ast,\ast,\ast,1,1,\ast)$ for example; if the second $1$ were paired to the second $\ast$, the substring so-contained would be $(1,1,\ast)$, which is not balanced and so has no internal pairings.  Therefore, in order for the whole string to be paired off, it is not possible for the second $1$ to pair to the second $\ast$.  To understand which pairings may be made, associate to any string $\mx{S}$ a {\em lattice path} $\mathscr{P}(\mx{S})$: start at the origin in $\R^2$, and for each $1$ in the string, draw a line segment of direction vector $(1,1)$; for each $\ast$ draw a line segment of direction vector $(1,-1)$.  If the balanced string $\mx{S}$ has length $2n$ ($n$ $1$s and $n$ $\ast$s), then the associated lattice path $\mathscr{P}(\mx{S})$ is a $\pm 1$-slope piecewise-linear curve joining $(0,0)$ to $(2n,0)$ (and each such curve, with slope-breaks at integer points, corresponds to a balanced $1$-$\ast$ string).  Figure \ref{fig lattice path} shows the lattice path corresponding to the string $\mx{S}$ considered above.
\begin{figure}[htbp]
\begin{center}
\input{fig5.pstex_t}
\caption{The lattice path $\mathscr{P}(1,1,1,1,\ast,\ast,1,1,\ast,\ast,\ast,\ast,\ast,1,1,\ast)$.}
\label{fig lattice path}
\end{center}
\end{figure}
The lattice path $\mathscr{P}(\mx{S})$ gives an easy geometric condition on allowed non-crossing pairings of $\mx{S}$.  Any $1$ must be paired with a $\ast$ in such a way that the substring between them is balanced; since the line-segments in $\mathscr{P}(\mx{S})$ slope up for $1$ and down for $\ast$, this means that any pairing must be from an up slope to a down slope {\em at the same vertical level}.  In Figure \ref{fig lattice path}, these levels are marked with dotted lines, and labeled along the vertical axis.

\medskip

The statistic of a string which is important for the upper bound is the {\em lattice path height} $h(\mx{S})$, which is simply the total height (total number of vertical increments) in $\mathscr{P}(\mx{S})$; that is, $h(\mx{S})$ is the number of distinct labels needed on the vertical axis of the lattice path.  In Figure \ref{fig lattice path}, $h=5$.

\begin{lemma} \label{lemma upper bound on NC_2} Let $\mx{S}$ be a balanced $1$-$\ast$ string with lattice path height $h = h(\mx{S})$ and $r$ runs.  Then
\begin{equation} \label{eqn Fuss-Catalan upper bound} |NC_2(\mx{S})| \le C^{(h)}_r \le \frac{r^{r-1}}{r!}\,(1+h)^{r-1}. \end{equation}
\end{lemma}

\begin{proof} This is a purely combinatorial fact, owing to a nice inclusion $NC_2(\mx{S}) \subseteq NC_2\left((1^h,\ast^h)^r\right)$, as follows.  In the lattice path $\mathscr{P}\left((1^h,\ast^h)^r\right)$, locally pair those peaks and troughs whose height-labels do not occur at the corresponding levels in the lattice path $\mathscr{P}(\mx{S})$.  (If the lattice path dips below the horizontal axis, ``locally'' may mean matching the first block to the last one; one could take care of this by first rotating the string so its minimal block is first.)  The remaining unpaired entries in $(1^h\,\ast^h)^r$ form a copy of $\mx{S}$, and since the labels correspond, there is a bijection between pairings of $\mx{S}$ and pairings of this inclusion of substrings.  This gives the inclusion, and the result follows from Equation \ref{Fuss-Catalan}.  Figure \ref{fig proof of lemma upper bound} demonstrates the inclusion.  \end{proof}
\begin{figure}[htbp]
\begin{center}
\input{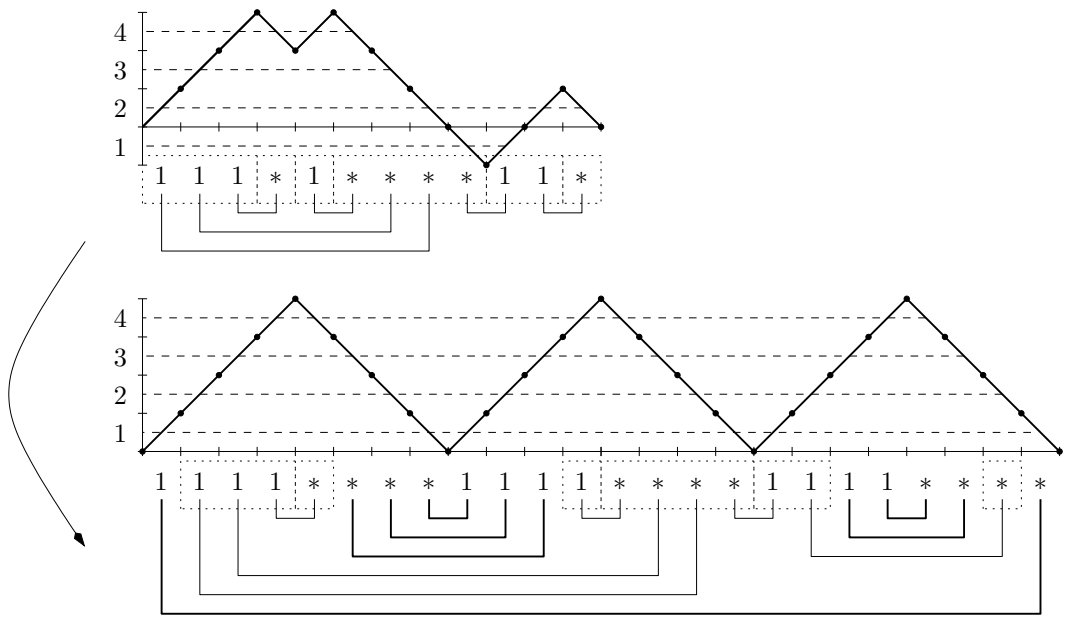}
\caption{ $\mx{S}$ is injected into $\mx{S}^h_r$, with extraneous labels (dark lines) paired locally.}
\label{fig proof of lemma upper bound}
\end{center}
\end{figure}
\begin{remark} Note that the lattice path height is the smallest $h$ which can be used in the proof of Lemma \ref{lemma upper bound on NC_2}, since all labels appearing in $\mathscr{P}(\mx{S})$ must be present in $\mathscr{P}\left((1^h,\ast^h)^r\right)$.  Unfortunately, $h(\mx{S})$ can be quite large in comparison to the average (or even maximum) block size in $\mx{S}$: consider the string $(1^{,k},\ast)^{,\ell},(1,\ast^{,k})^{,\ell}$.  The maximum block size is $k$, while the lattice path height is $(k-1)\ell + 1$.  Indeed, this string has length $2(k+1)\ell$, and so the height is about half the total length.  In general, this is about the best that can be said, and so the only generally useful corollary is the following. \end{remark}

\begin{corollary} \label{cor upper bound NC_2} Let $\mx{S}$ be a balanced string of length $2n$ ($n$ $1$s and $n$ $\ast$s), with $r$ runs.  Then
\begin{equation}\label{eqn upper bound NC_2} |NC_2(\mx{S})| \le \frac{r^{r-1}}{r!}\, (1+n)^{r-1}. \end{equation}
\end{corollary}

\begin{proof} From Lemma \ref{lemma upper bound on NC_2}, it suffices to show that if $\mx{S}$ has $n$ $1$s then $h(\mx{S}) \le n$.  Let us rotate $\mx{S}$ so that its minimum block is first.  This means that the lattice path $\mathscr{P}(\mx{S})$ never drops below the horizontal axis, and so each vertical label corresponds to at least one up-slope (the first one where it appears, for example).  This means that the lattice path height $h(\mx{S})$ is bounded above by the number of up slopes, which is $n$.  This bound is achieved only when $r=1$.
\end{proof}

\begin{remark} The bound in Corollary \ref{cor upper bound NC_2} is quite large and essentially never achieved.  A better bound, proved in our paper \cite{KMRS}, replaces $h$ in Equation \ref{eqn Fuss-Catalan upper bound} with the average block size; in view of Equation \ref{Fuss-Catalan}, this is the best possible general bound.  Its proof is very involved, and the very rough estimate of Equation \ref{eqn upper bound NC_2} is sufficient for our purposes in what follows. \end{remark}

\medskip

For small $r$, it is possible to give explicit formulas for the sizes $|NC(\mx{S})|$ and $|NC_2(\mx{S})|$ for all strings $\mx{S}$.  Following are such formulas in the case $r=2$, which we will use in Section \ref{optimal ultracontractivity} below.

\begin{theorem} \label{2 runs theorem} Let $\mx{S} = (1^{n_1},\ast^{m_1},1^{n_2},\ast^{m_2})$ be any balanced string with $2$ runs (so $n_1+n_2 = m_1+m_2$).  Then
\[ \begin{aligned} |NC_2(\mx{S})| &= 1+\min\{n_1,m_1,n_2,m_2\}, \\
|NC(\mx{S})| &= 1+2\min\{n_1,m_1,n_2,m_2\}. \end{aligned} \]
\end{theorem}

\begin{proof} If $\pi\in NC(\mx{S})$, then each of its blocks alternates between $1$s and $\ast$s; from the form of $\mx{S}$ this means blocks must have length at most $4$, so $\pi$ consists of $2$-blocks and $4$-blocks.
If $\pi$ has a $4$-block, all remaining blocks of $\pi$ are nested between its consecutive pairings, and since a $4$-block must connect one element from each of the four segments in $\mx{S}$, it follows from the non-crossing condition that all remaining blocks of $\pi$ are $2$-blocks, and are determined by the position of the $4$-block.  This is demonstrated in Figure \ref{fig 4-block}.
\begin{figure}[htbp]
\begin{center}
\input{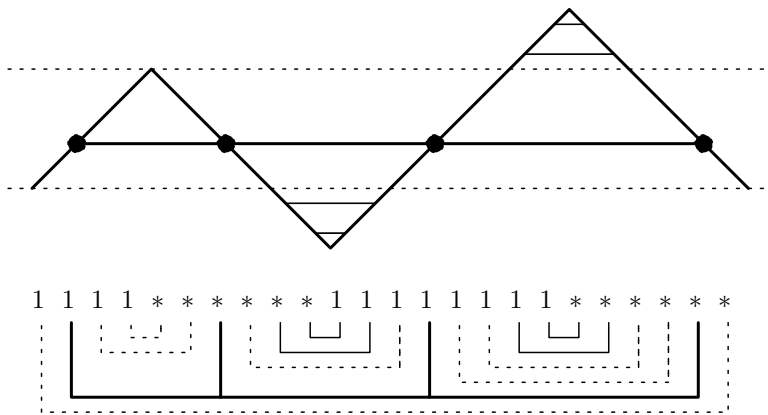}
\caption{The $4$-block (dark) determines all other blocks of the partition in this 2-run string.  The high and low peaks must be paired locally since there are no other choices at the same level; the $4$-block can be placed at any height within the first (minimal) string of $1$s, and once in place it determines all other pairings: local inside, and far outside.  These pairings are dotted in the partition diagram.}
\label{fig 4-block}
\end{center}
\end{figure}
Hence, the number of $\pi\in NC(\mx{S})$ with a $4$-block is equal to the number of positions the $4$-block can occupy.  If we rotate $\mx{S}$ so that $n_1 = \min\{n_1,m_1,n_2,m_2\}$ as in Figure \ref{fig 4-block} then it is clear this number is precisely $n_1$: the height of the $4$-block must be within the smallest block $n_1$, and this height determines the rest of the partition. 

\medskip

If, on the other hand, $\pi$ contains no $4$-blocks, then it is in $NC_2(\mx{S})$.  Rotate again so $n_1$ is the minimum block size.  Since pairings must be made at the same vertical level in the lattice diagram of $\mx{S}$, all the points below the axis or above $n_1$ in height must be paired locally.  The unpaired entries of $\mx{S}$ form a regular string $(1^{n_1},\ast^{n_1})^2$ whose pairings number $C^{(n_1)}_2 = 1+n_1$.  This proves the result. \end{proof}

\begin{remark} A formula in the case of $3$ runs can also be written down with a little more difficulty.  In fact, there is a general formula for $|NC_2(\mx{S})|$ expressed as a sum of binomial coefficients, index but rooted trees determined by the string $\mx{S}$; see \cite{KMRS}. \end{remark}

\subsection{$L^p$-bounds of $D_t$} 

Let $p$ be a positive even integer, $p=2r$, and let $a_1,\ldots,a_d$ be $\ast$-free $\mathscr{R}$-diagonal operators in $(\mathscr{A},\ff)$.  Our goal in this section is to give optimal bounds on the norm of an $\mathscr{R}$-diagonal dilation semigroup $D_t\colon L^2_{hol}(a_1,\ldots,a_d)\to L^p(\mathscr{A},\ff)$.  In \cite{Kemp Speicher}, we proved such a bound in the case $p=\infty$, through an application of our strong Haagerup inequality.  In the circular case (generators $c_1,\ldots,c_d$ free $\ast$-free circular operators), the same techniques we developed in that paper demonstrate the following estimate holds for small $t>0$:
\[ \|D_t\colon L^2_{hol}(c_1,\ldots,c_d)\to L^p(\mathscr{A},\ff)\| \le \alpha_p\, t^{-1+\frac{1}{p}}. \]
One can check from the asymptotics of the Fuss-Catalan numbers that $\|c^n\|_p \sim n^{\frac{1}{2}-\frac{1}{p}}$, and as a result the above estimate cannot be improved by means of a Haagerup inequality approach similar to that in Section 5 of \cite{Kemp Speicher}.  Nevertheless, this estimate is {\em not} optimal.  In the case of a single circular generator $c$, the correct bound is
\begin{equation} \label{hyp bounds} 
\|D_t\colon L^2_{hol}(c)\to L^p(W^\ast(c),\ff)\| \sim t^{-1+\frac{2}{p}},
\end{equation}
for even $p$ and $t>0$. As a lower-bound, this holds more generally for any $\mathscr{R}$-diagonal generator with non-negative free cumulants, or any {\em infinite} $\ast$-free $\mathscr{R}$-diagonal generating set (through a straightforward application of the free central limit theorem).  We proceed to prove these bounds in the following; first, we begin with a well-known estimate which will be key to the proofs.

\begin{lemma} \label{lemma sum exp}  Let $0 \le q < \infty$.  There are constants $\alpha_q,\beta_q>0$ such that, for $0<t<1$,
\[ \alpha_q\, t^{-q-1} \le \sum_{n\ge 0} n^q\,e^{-nt} \le \, \beta_q\, t^{-q-1}. \]
\end{lemma}

\begin{proof} For the lower bound, for any integer $k\ge 1$ look at just those terms $n$ that lie strictly between $\frac{1}{(k+1)t}$ and $\frac{1}{kt}$; since the difference is $\frac{1}{t}\frac{1}{k(k+1)}$, there are at least $\frac{1}{t}\frac{1}{(k+1)^2}$ such terms.  For each one, $n \le \frac{1}{kt}$ and so $e^{-nt} \ge e^{-1/k}$, while $n^q \ge \frac{1}{((k+1)t)^q}$.  So in total, we have
\[ \sum_{\frac{1}{(k+1)t} \le n \le \frac{1}{kt}} n^q e^{-nt} \ge e^{-1/k} \frac{1}{t}\frac{1}{(k+1)^2}\frac{1}{((k+1)t)^q} = e^{-1/k} \frac{1}{(k+1)^{q+2}}\; t^{-q-1}. \]
The largest such term is achieved with $k=1$, but we can add them up; since $e^{-1/k} \ge e^{-1}$, this yields that the sum over all $n\ge 0$ is at least $e^{-1}\sum_{k\ge 1} \frac{1}{(k+1)^{q+2}}\;t^{-q-1}$, which yields the results.

\medskip

As to the upper bound, note that the function $x\mapsto x^p e^{-x/2}$ is bounded for $x\ge 0$, with maximum value $(2q)^q e^{-q}$ achieved at $x=2q$.  Therefore $x^q e^{-x} \le (2q)^qe^{-q}\,e^{-x/2}$, and plugging in $x = nt$ we have
\[ \sum_{n\ge 0} (nt)^q e^{-nt} \le (2q)^q e^{-q} \sum_{n\ge 0} e^{-nt/2} = \frac{(2q)^q e^{-q}}{1-e^{-t/2}}. \]
The function $t/(1-e^{-t/2})$ is bounded near $0$ and increasing on $(0,1)$; at $1$ its value is $<3$, and so  $(1-e^{-t/2})^{-1} \le 3/t$ on the interval $(0,1)$, yielding the result.
\end{proof}

Following is the upper-bound half of Equation \ref{hyp bounds}.

\begin{theorem} \label{Lp upper bound} Let $T\in L^2_{hol}(c)$, and let $D_t$ denote the associated $\mathscr{R}$-diagonal dilation semigroup (in this case, $D_t$ is the free O--U semigroup).  For each even $p\ge 4$, there is a constant $\alpha_p$ so that, for $0<t<1$,
\[ \|D_t\,T\|_p \le \alpha_p\,t^{-1+\frac{2}{p}}\,\|T\|_2. \]
\end{theorem}

\begin{remark} A version of the following proof works for any finite circular generating set as well, but is significantly more complicated.  We do not include it here since the $p=\infty$ case, which is our main interest, is proved independently in Section \ref{optimal ultracontractivity}.  \end{remark}

\begin{proof} Let $T = \sum_n \l_n c^n$.  By Corollary \ref{cor orthogonal}, this is an orthogonal sum.  Also,
\[ \|c^n\|_2^2 = \ff(c^n c^{\ast n}) = \sum_{\pi\in NC_2(\mx{S})} \k_\pi[c,\ldots,c,c^\ast,\ldots,c^\ast] \]
by Equation \ref{R-diag moment--cumulant} and the fact that only $\k_2\ne 0$ for $c$, where $\mx{S} = (1^n,\ast^n)$.  This is a regular pattern treated by Equation \ref{Fuss-Catalan}, and so the number of such $\pi$ is $C^{(n)}_1 = 1$.  Indeed, the only element of $NC(\mx{S})$ is the fully nested pairing $\varpi$:
\begin{figure}[htbp]
\begin{center}
\input{nested.pstex_t}
\caption{The only partition in $NC(1,\ldots,1,\ast,\ldots,\ast)$.}
\label{fig nested}
\end{center}
\end{figure}
(All partitions in $NC(\mx{S})$ are pairing, since each block of such a partition must alternate between $1$ and $\ast$, and all $\ast$s in $\mx{S}$ follow all $1$s.)  Thus, $\|c^n\|_2^2 = \k_\varpi[c,\ldots,c,c^\ast,\ldots,c^\ast] = (\k_2[c,c^\ast])^n = 1$, and we have
\begin{equation} \label{2 norm} \|T\|_2^2 = \sum_n |\l_n|^2. \end{equation}
Now, let $p=2r$ for $r\in\N$.  Note that $D_t\,T = \sum_n e^{-nt}\l_n c^n$, and so
\[ [(D_t\,T) (D_t\,T)^\ast]^r = \sum_{\nn,\mm}  e^{-(|\nn|+|\mm|)t} \l_\nn \overline{\l_\mm}\, c^{n_1}c^{\ast m_1}\cdots c^{n_r}c^{\ast m_r}, \]
 where $\nn=(n_1,\ldots,n_r)$ and $\mm=(m_1,\ldots,m_r)$ range independently over $\N^r$, and $\l_\nn = \l_{n_1}\cdots\l_{n_r}$.  From Corollary \ref{cor balanced}, $\ff(c^{n_1}c^{\ast m_1}\cdots c^{n_r}c^{\ast m_r}) = 0$ unless $|\nn|=n_1+\cdots+n_r = m_1+\cdots+m_r = |\mm|$, and so
\begin{equation} \label{p norm 1} \|D_t\,T\|_p^p = \sum_{n=0}^\infty e^{-2nt} \sum_{|\nn|=|\mm|=n}  \l_\nn\overline{\l_\mm}\;\ff(c^{n_1}c^{\ast m_1}\cdots c^{n_r}c^{\ast m_r})). \end{equation}
\noindent From Equation \ref{R-diag moment--cumulant}, $\ff(c^{n_1}c^{\ast m_1}\cdots c^{n_r}c^{\ast m_r})) = \sum_{\pi\in NC(\mx{S}_{\nn,\mm})} \k_\pi[c^{n_1},c^{\ast m_1},\ldots,c^{n_r},c^{\ast m_r}]$ where $\mx{S}_{\nn,\mm}=(1^{n_1},\ast^{m_1},\ldots,1^{n_r},\ast^{m_r})$. Since only $\k_2\ne 0$ for circular variables, the sum is really over $NC_2(\mx{S}_{\nn,\mm}$, and all such terms are equal to $1$ since $\k_2[c,c^\ast]=1$.  Hence
\begin{equation} \label{p norm 2} \|D_t\,T\|_p^p = \sum_{n=0}^\infty e^{-2nt} \sum_{|\nn|=|\mm|=n}  \l_\nn\overline{\l_\mm}\,|NC_2(\mx{S}_{\nn,\mm})|. \end{equation}
We now employ the estimate of Equation \ref{eqn upper bound NC_2}:
\[ |NC_2(\mx{S}_{\nn,\mm})| \le \frac{r^{r-1}}{r!}\, (1+n)^{r-1}. \]
Thus,
\begin{equation} \label{p norm 3}
\|D_t\,T\|_p^p \le \frac{r^{r-1}}{r!}\, \sum_{n=0}^\infty\, (1+n)^{r-1} e^{-2nt} \sum_{|\nn|=|\mm|=n}  \l_\nn\overline{\l_\mm}. \end{equation}
Note that
\[ \sum_{|\nn|=|\mm|=n} \l_\nn \overline{\l_\mm} = \big|\sum_{|\nn|=n} \l_\nn \big|^2, \]
and using the optimal $\ell^1$--$\ell^2$ estimate $|\sum_{k\in S} a_k|^2 \le |S| \sum_{k\in S} |a_k|^2$ and the fact that the set of ordered integer partitions $\{\nn\in\N^r\,;\,|\nn|=n\}$ is counted by the binomial coefficient $\binom{n+r-1}{r-1}$, we have
\[ \big|\sum_{|\nn|=n} \l_\nn \big|^2 \le \binom{n+r-1}{r-1} \sum_{|\nn|=n} |\l_\nn|^2. \]
Combining with Equation \ref{p norm 3} yields
\begin{equation} \label{p norm 4}
\|D_t\,T\|_p^p \le \frac{r^{r-1}}{r!}\, \sum_{n=0}^\infty\, (1+n)^{r-1}\binom{n+r-1}{r-1} e^{-2nt} \sum_{|\nn|=n} |\l_\nn|^2. \end{equation}
Let $S(t) = \sup_{n\ge 0} (1+n)^{r-1}\binom{n+r-1}{r-1}e^{-2nt}$; then we can roughly estimate Equation \ref{p norm 4} by
\[ \|D_t\,T\|_p^p \le \frac{r^{r-1}}{r!}S(t)\, \sum_{n=0}^\infty \sum_{|\nn|=n} |\l_n|^2, \]
and this latter sum $\sum_{\nn} |\l_\nn|^2 = (\sum_n |\l_n|^2)^r$ is simply $\|T\|_2^{2r} = \|T\|_2^{p}$ from Equation \ref{2 norm}.  We are left, therefore, to estimate $S(t)$.  Note that
\[ \binom{n+r-1}{r-1} = \frac{(n+r-1)\cdots(n+1)}{(r-1)!} \le \frac{(n+r-1)^{r-1}}{(r-1)!}, \]
and so estimating $n+1 \le n+r-1$, 
\[\begin{aligned} S(t) &\le \frac{1}{(r-1)!}\;\sup_{n\ge 0}\; (n+r-1)^{2(r-1)} e^{-2nt} \\
&= \frac{1}{(r-1)!}\;\left(\sup_{n\ge 0}\, (n+r-1)\, e^{-\frac{n}{r-1}t}\right)^{2(r-1)}. \end{aligned} \]
Elementary calculus yields that the supremum over all real $n\ge 0$ of $(n+r-1)e^{-\frac{n}{r-1}t}$ is $\frac{r-1}{t} e^{t-1}$, and so
\[ S(t) \le \frac{(r-1)^{2(r-1)}}{(r-1)!}\,e^{2(r-1)(t-1)}\,t^{-2r+2}. \]
Altogether, then, we have
\begin{equation} \label{p norm 5}
\|D_t\,T\|_p^p \le \frac{(r(r-1)^2)^{r-1}}{r!(r-1)!} e^{2(r-1)(t-1)} t^{-p+2}\,\|T\|_2^p.
\end{equation}
For $0<t<1$, $e^{2(r-1)(t-1)} \le 1$, and so taking $p$th roots and letting $\alpha_p$ represent the ratio following the $\le$ in Equation \ref{p norm 5} yields the desired result.
\end{proof}

\begin{remark} The constant $\alpha_p$ developed through Equation \ref{p norm 5} tends to $\infty$ as $p\to\infty$.  However, the $p=\infty$ case of Theorem \ref{Lp upper bound} is actually a special case of Theorem 5.4 in \cite{Kemp Speicher}, following a different technique. \end{remark}

We now turn to the lower-bound in Equation \ref{hyp bounds}, and show that it is optimal in considerable generality.

\begin{theorem} \label{Lp lower bound k>0} Let $a_1,\ldots,a_d$ be $\ast$-free $\mathscr{R}$-diagonal operators, and suppose that $a_1$ has non-negative free cumulants: for each $n$, $\k_{2n}[a_1,a_1^\ast,\ldots,a_1,a_1^\ast] \ge 0$ and $\k_{2n}[a_1^\ast,a_1,\ldots,a_1^\ast,a_1]\ge 0$.  Then for each even integer $p\ge 2$, there is a constant $\alpha_p>0$ such that, for $0<t<1$,
\[ \|D_t\colon L^2_{hol}(a_1,\ldots,a_d)\to L^p(W^\ast(a_1,\ldots,a_d),\ff)\| \ge \alpha_p\,t^{-1+\frac{2}{p}}. \]
Moreover, this bound is achieved on the subspace $L^2_{hol}(a_1)\subset L^2_{hol}(a_1,\ldots,a_d)$.
\end{theorem}

\begin{proof} Set $p=2r$, and denote $a_1$ by $a$. We will show that $t^{1-1/r}\cdot \| D_t \colon  L^2_{hol}(a)\to L^{2r}(W^\ast(a),\phi) \|$ is bounded above $0$ for small $t$.  In fact, we will show that for each small $t>0$, there is an element $\psi_t\in L^2_{hol}(a)$ so that $t^{2r-2} \|D_t\,\psi_t\|_{2r}^{2r} / \|\psi_t\|_2^{2r} \ge \alpha_r$ for a $t$-independent constant $\alpha_r > 0$.  Indeed, for fixed $t>0$ define
\[ \psi_t = \sum_{n\ge0} e^{-nt} \;{a}^n, \]
where we rescale $a$ so that $\|a\|_2=1$. (Formally $\psi_t$ is $D_t \psi$ where $\psi = \sum_{n\ge0} a^n = (1-a)^{-1}$ if $1-a$ is invertible in $L^2_{hol}(a)$.)  Thus $D_t\,\psi_t = \psi_{2t}$, and so we wish to consider the ratio $\|\psi_{2t}\|_{2r}^{2r}/\|\psi_t\|_2^{2r}$.  We begin by expanding the numerator.
\begin{equation} \label{eqn expand psi_t 1}
\|\psi_{2t}\|_{2r}^{2r} = \p[(\psi_{2t} \psi_{2t}^\ast)^r] = \p\sum_{n_1,\ldots,n_r\ge 0 \atop m_1,\ldots,m_r \ge 0}
e^{-2n_1t}\; a^{n_1} \, e^{-2m_1 t}\;a^{\ast m_1} \cdots e^{-2n_r t}\;a^{n_r}\, e^{-2m_r t} \;a^{\ast m_r}. \end{equation}
It is convenient to add two more summation indices: $n = n_1+\cdots+n_r$ and $m=m_1+\cdots+m_r$.  Equation \ref{eqn expand psi_t 1} then becomes
\begin{equation} \label{eqn expand psi_t 2}
\|\psi_{2t}\|_{2r}^{2r}= \sum_{n,m\ge 0}  e^{-2(n+m)t} \sum_{n_1,\ldots,n_r\ge 0 \atop n_1+\cdots+n_r = n} \sum_{m_1,\ldots,m_r\ge0 \atop m_1+\cdots+m_r =m} \p(a^{n_1} a^{\ast m_1}\cdots a^{n_r} a^{\ast m_r}).
\end{equation}
Referring back to Corollary \ref{cor balanced}, the only non-zero terms in Equation \ref{eqn expand psi_t 2} are those for which the word $a^{n_1}a^{\ast m_1}\cdots a^{n_r}a^{\ast m_r}$ is balanced:\ there must be as many $a$s as $a^\ast$s, and so we must have $n_1+\cdots+n_r = m_1+\cdots+m_r$; i.e.\ $n=m$.  Thus
\begin{equation} \label{eqn expand psi_t 3}
\|\psi_{2t}\|_{2r}^{2r}= \sum_{n\ge 0}  e^{-4nt} \sum_{n_1,\ldots,n_r\ge 0 \atop n_1+\cdots+n_r = n} \sum_{m_1,\ldots,m_r\ge0 \atop m_1+\cdots+m_r =n} \p(a^{n_1} a^{\ast m_1}\cdots a^{n_r} a^{\ast m_r}).
\end{equation}
Now, let $\mx{S}_{m_1,\ldots,m_r}^{n_1,\ldots,n_r}$ denote the string $(1^{n_1},\ast^{m_1},\ldots,1^{n_r},\ast^{m_r})$.  Equation \ref{R-diag moment--cumulant} then yields that
\begin{equation} \label{eqn lower bound cumulant expansion 1} \p(a^{n_1} a^{\ast m_1}\cdots a^{n_r} a^{\ast m_r}) = \sum_{\pi\in NC(\mx{S}_{m_1,\ldots,m_r}^{n_1,\ldots,n_r})} \k_\pi[a^{,n_1},a^{\ast,m_1},\ldots,a^{,n_r},a^{\ast,m_r}]. \end{equation}
(The formerly implied commas are now explicit in the exponents; this is to make clear that the free cumulants have $2n$ arguments and are {\em not} being evaluated at products of arguments.) By assumption, all cumulants in $a,a^\ast$ are $\ge 0$, and so we may restrict the summation in Equation \ref{eqn lower bound cumulant expansion 1} to those $\pi$ that are pairings.
\begin{equation} \label{eqn lower bound cumulant expansion 2} \p(a^{n_1} a^{\ast m_1}\cdots a^{n_r} a^{\ast m_r}) \ge \sum_{\pi\in NC_2(\mx{S}_{m_1,\ldots,m_r}^{n_1,\ldots,n_r})} \k_\pi[a^{,n_1},a^{\ast,m_1},\ldots,a^{,n_r},a^{\ast,m_r}]. \end{equation}
Now, for each pairing $\pi\in NC_2(\mx{S}^{n_1,\ldots,n_r}_{m_1,\ldots,m_r})$, each block $\pi$ matches an $a$ with an $a^\ast$; there are $n$ such blocks in total, and each is $\k_2[a,a^\ast]$ or $\k_2[a^\ast,a]$.  Since $a$ is $\mathscr{R}$-diagonal, $\p(a) = \k_1[a] = 0$ and $\p(a^\ast) = \k_1[a^\ast] = 0$; thus $\k_2[a,a^\ast] = \p(aa^\ast)-\p(a)\p(a^\ast) = \|a\|_2^2=1$, and $\k_2[a^\ast,a] = \p(a^\ast a)-\p(a^\ast)\p(a)=\p(a^\ast a) = \|a\|_2^2=1$ by the traciality assumption.  In general, then, we have \[ \k_\pi[a^{,n_1},a^{\ast,m_1},\ldots,a^{,n_r},a^{\ast,m_r}] = \|a\|_2^{2n}=1. \]
Combining this with Equation \ref{eqn lower bound cumulant expansion 1} and Equation \ref{eqn expand psi_t 3} we get
\begin{equation} \label{eqn expand psi_t 4}
\|\psi_{2t}\|_{2r}^{2r} \ge \sum_{n\ge 0}  e^{-4nt} \sum_{n_1,\ldots,n_r\ge 0 \atop n_1+\cdots+n_r = n} \sum_{m_1,\ldots,m_r\ge0 \atop m_1+\cdots+m_r =n} |NC_2(\mx{S}^{n_1,\ldots,n_r}_{m_1,\ldots,m_r})|.
\end{equation}
We will now throw away all of the terms where any index $n_j$ or $m_j$ is $0$; in this case, since each of the $r$ indices $n_j$ (or $m_j$) is at least $1$, their sum $n$ must be at least $r$, and so we have
\begin{equation} \label{eqn expand psi_t 5}
\|\psi_{2t}\|_{2r}^{2r} \ge \sum_{n\ge r}  e^{-4nt} \sum_{n_1,\ldots,n_r\ge 1 \atop n_1+\cdots+n_r = n} \sum_{m_1,\ldots,m_r\ge 1 \atop m_1+\cdots+m_r =n} |NC_2(\mx{S}^{n_1,\ldots,n_r}_{m_1,\ldots,m_r})|.
\end{equation}
 Now, fix $n$ and let us reorganize the internal summation.  As the indices $n_1,\ldots,n_r$ and $m_1,\ldots,n_r$ range over their summation sets, the string $\mx{S}^{n_1,\ldots,n_r}_{m_1,\ldots,m_r}$ ranges over all possible balanced $1$-$\ast$ strings (beginning with $1$ and ending with $\ast$) with length $2n$ and $r$ runs.  Let us denote this set of strings by $\Omega^n_r$.  Then the internal sum in Equation \ref{eqn expand psi_t 5} can be rewritten as
\begin{equation} \label{eqn sum over Omega 1}
\sum_{n_1,\ldots,n_r\ge 0 \atop n_1+\cdots+n_r = n} \sum_{m_1,\ldots,m_r\ge0 \atop m_1+\cdots+m_r =n} |NC_2(\mx{S}^{n_1,\ldots,n_r}_{m_1,\ldots,m_r})|= \sum_{\mx{S}\in\Omega^n_r} |NC_2(\mx{S})|.
\end{equation}

We can break up the set $\Omega_r^n$ according to the size of the minimal block in each element.  Let $\Omega_r^{n,i}$ denote the subset of $\Omega_r^n$ of all balanced, length $2n$, $r$ run strings with minimal block size $i$.  (Note: this set is empty unless $n\ge ir$.)  Evidently the sets $\Omega_r^{n,i}$ are disjoint for different $i$, and indeed $\Omega_r^n = \bigsqcup_{i=1}^{n/r} \Omega_r^{n,i}$.  Combining this with Equations \ref{eqn sum over Omega 1} and \ref{eqn expand psi_t 5}, and reordering the sum, this yields
\begin{equation} \label{eqn sum over Omega 2}
\|\psi_{2t}\|_{2r}^{2r} \ge \sum_{i=1}^\infty \sum_{n=ir}^\infty e^{-4nt} \sum_{\mx{S}\in\Omega_r^{n,i}} |NC_2(\mx{S})|. \end{equation}
Now employing Proposition \ref{prop NC2 lower bound}, we have that for each $\mx{S}\in\Omega_r^{n,i}$, $|NC_2(\mx{S})| \ge (1+i)^{r-1}$.  Hence
\begin{equation} \label{eqn sum over Omega 3}
\|\psi_{2t}\|_{2r}^{2r} \ge \sum_{i=1}^\infty (1+i)^{r-1} \sum_{n=ir}^\infty e^{-4nt}\; |\Omega^{n,i}_r|. \end{equation}
For fixed $i,n$ it is relatively straightforward to enumerate the set $\Omega^{n,i}_r$.  Let $\mx{S}\in\Omega^{n,i}_r$; then $\mx{S}$ can be written as $\mx{S}^{n_1,\ldots,n_r}_{m_1,\ldots,m_r}$ for indices $n_j,m_j$ satisfying $n_1+\cdots+n_r = m_1+\cdots+m_r = n$ and $n_j,m_j\ge i$, and where at least one of $n_1,\ldots,m_r$ is equal to $i$.  We will consider here only those terms for which $n_1=i$.  Let $n_j' = n_j - i$ and $m_j' = m_j-i$; then we can rewrite $\mx{S}$ as the string $(1^{0+i},\ast^{m_1'+i},\ldots,1^{n_r'+i},\ast^{m_r'+i})$, where the $n_j',m_j'$ are $\ge 0$ and sum to $n-ri$.  That is, there is an injection
\begin{equation} \label{eqn bijection Omega} \Omega_r^{n,i} \hookleftarrow \{(0,m_1',\ldots,n_r',m_r')\,;\,\forall j\;n_j',m_j'\ge 0 \;\;\& \;\;n_2'+\cdots+n_r' = m_1'+m_2'+\cdots+m_r' = n-ir\}.
\end{equation}
The set on the right-hand-side of Equation \ref{eqn bijection Omega} is a Cartesian product of the (ordered) integer partition sets $\{(n_2',\ldots,n_r')\,;\, \forall j\; n_j'\ge 0 \;\; \& \;\; n_2'+\cdots+n_r' = n-ir\}$ and $\{(m_1',\ldots,n_r')\,;\, \forall j\; n_j'\ge 0 \;\; \& \;\; m_1'+\cdots+m_r' = n-ir\}$; they have sizes $\binom{n-ir+r-2}{r-2}$ and  $\binom{n-ir+r-1}{r-1}$ respectively.  Thus, Equation \ref{eqn sum over Omega 3} yields
\begin{equation} \label{eqn expand psi_t 6}
\|\psi_{2t}\|_{2r}^{2r} \ge \sum_{i=1}^\infty (1+i)^{r-1} \sum_{n=ir}^\infty e^{-4nt} \; \binom{n-ir+r-2}{r-2}\binom{n-ir+r-1}{r-1}. \end{equation}
We can lower bound the binomial coefficient as follows.
\[ \binom{n-ir + r-1}{r-1} = \frac{(n-ir+r-1)(n-ir+r-2)\cdots(n-ir+1)}{(r-1)!} \ge \frac{(n-ir)^{r-1}}{(r-1)!}, \]
and similarly $\binom{n-ir+r-2}{r-2} \ge (n-ir)^{r-2}/(r-2)!$.
Combining with Equation \ref{eqn expand psi_t 6} this yields
\begin{equation} \label{eqn expand psi_t 7}
\|\psi_{2t}\|_{2r}^{2r} \ge \frac{1}{(r-2)!(r-1)!} \sum_{i=1}^\infty (1+i)^{r-1} \sum_{n=ir}^\infty (n-ir)^{2r-3}\,e^{-4nt} . \end{equation}
Reindexing the internal sum, we have
\[ \sum_{n=ir}^\infty (n-ir)^{2(r-1)}\,e^{-4nt} = \sum_{n\ge 0} n^{2r-3} e^{-4(n+ir)t} = e^{-4irt}\sum_{n\ge 0} n^{2r-3} e^{-4nt}. \]
Appealing to Lemma \ref{lemma sum exp}, the sum $\sum_{n\ge 0} n^{2(r-1)} e^{-4nt} \ge  \alpha_r\,(4t)^{-(2r-3)-1}$ for small $t$.  Combining with Equation \ref{eqn expand psi_t 7}, this give us
\begin{equation} \label{eqn expand psi_t 8}
\|\psi_{2t}\|_{2r}^{2r} \ge \alpha_r \, t^{-2r+2} \sum_{i=1}^\infty (1+i)^{r-1} e^{-4irt}.
\end{equation}
(As before, $\alpha_r$ is used for an arbitrary $r$-dependent constant.) Applying Lemma \ref{lemma sum exp} again, the remaining summation may be estimated
\begin{equation} \label{eqn final exp sum} \begin{aligned}
\sum_{i=i}^\infty (1+i)^{r-1} e^{-4irt} \ge \sum_{i=1}^\infty (i-1)^{r-1} e^{-4irt} &= \sum_{i\ge 0} i^{r-1} e^{-4(i+1)rt} \\
& \ge e^{-4rt}\cdot\alpha_r\, t^{-(r-1)-1}.  \end{aligned} \end{equation}
\noindent For $0<t<1$, we have $e^{-4rt} > e^{-4r}$.  Collecting all constants and combining Equations \ref{eqn expand psi_t 8} and \ref{eqn final exp sum}, we have
\begin{equation} \label{eqn expand psi_t 9}
\|\psi_{2t}\|_{2r}^{2r} \ge \alpha_r\, t^{-3r+2}. \end{equation}

\medskip

Now turning to the denominator, since different powers of $a$ are orthogonal (by Corollary \ref{cor orthogonal}) we have
\begin{equation} \label{2 norm a} \|\psi_t\|_2^2 = \sum_{n\ge 0} e^{-2nt}\, \|a^n\|_2^2, \end{equation}
and an analysis completely analogous to that leading up to Equation \ref{2 norm} (coupled with the normalization $\|a\|_2=1$) yields $\|a^n\|_2^2 = 1$ as well.  Thus, $\|\psi_t\|_2^2 = \sum_{n\ge 0} e^{-2nt} = (1-e^{-2t})^{-1}$, and so from Equation \ref{eqn expand psi_t 9} we have
\begin{equation} \label{eqn ratio} t^{2r-2}\cdot\frac{\|\psi_{2t}\|_{2r}^{2r}}{\|\psi_t\|_2^{2r}} \ge t^{2r-2}\cdot \alpha_r t^{-3r+2}\cdot (1-e^{-2t})^r = \alpha_r (1-e^{-2t})^r\, t^{-r}.
 \end{equation}
The function $t\mapsto (1-e^{-2t})/t$ is bounded and decreasing on $(0,1)$, and so is $\ge (1-e^{-2})$ on this interval.  We therefore have that $t^{2r-2}\cdot \|\psi_{2t}\|_{2r}^{2r} / \|\psi_t\|_2^{2r} \ge \alpha_r$ for $0<t<1$.  Taking $2r$th roots, this means that for each $t$ we have
\[   \frac{\|D_t\,\psi_t\|_{2r}}{\|\psi_t\|_2} = \frac{\|\psi_{2t}\|_{2r}}{\|\psi_t\|_2} \ge \alpha_r\, t^{-1+\frac{1}{r}}. \]
Since $\psi_t \in L^2_{hol}(a)$ for each $t>0$, this proves the theorem.  \end{proof}

\begin{remark} While the constant $\alpha_p$ in the above calculation is strictly positive, it decreases exponentially fast to $0$ as $r\to\infty$.  Hence, Theorem \ref{Lp lower bound k>0} does not yield the optimal ultracontractive bound in \cite{Kemp Speicher}, but rather a slightly weaker statement: it follows that $\|D_t\colon L^2_{hol}(a_1,\ldots,a_d)\to W^\ast(a_1,\ldots,a_d)\| \ge \alpha_\e\,t^{-1+\e}$ for any $\e>0$.  The fully optimal sharp bound does hold true, however, and also does not require the stringent non-negative cumulant condition of Theorem \ref{Lp lower bound k>0}.  For this purpose, we require an alternate technique which s is the subject of Section \ref{optimal ultracontractivity}. \end{remark}

\begin{remark} The condition of non-negative free cumulants is not entirely superfluous, as can easily be seen in the example of a single Haar unitary generator (some of whose free cumulants are negative).  In this case, an analysis similar to the proof of Theorem \ref{Lp upper bound} yields an $L^2\to L^p$ bound of order $t^{-\frac{1}{2} + \frac{1}{p}}$, which is optimal. \end{remark}

\begin{theorem} \label{CLT} Let $a_1,a_2,\ldots$ be an {\em infinite} set of $\ast$-free $\mathscr{R}$-diagonal operators, and let $\mathscr{A}= W^\ast(a_1,a_2,\ldots)$.  Let $p\ge 2$ be an even integer. Then there is a constant $\alpha_p>0$ so that, for $0<t<1$,
\[ \|D_t\colon L^2_{hol}(a_1,a_2,\ldots)\to L^p(\mathscr{A},\ff)\| \ge \alpha_p\,t^{-1+\frac{2}{p}}. \]
\end{theorem}

\begin{proof} This is an application of the free central limit theorem due to R.\ Speicher, \cite{Speicher 1}.  Let $A = \{a_1,a_2,\ldots\}$, where the generators are renormalized so that $\|a_j\|_2=1$ for all $j$.  Since $a_j$ is $\mathscr{R}$-diagonal, we then have $\p(a_j) = \p(a_j^\ast) = 0$, and $\p(a_j^\ast a_j) = \p(a_ja_j^\ast) = \|a_j\|_2^2 = 1$ for each $j$, while $\p(a_j^2) = \p(a_j^{\ast 2}) = 0$ (thanks to Corollary \ref{cor balanced}).  Hence, by Theorem 3 (and more specifically the remark following Theorem 6) in \cite{Speicher 1}, the sequence of elements
\[ a^{(N)} = \frac{1}{\sqrt{N}}(a_1 + \cdots + a_N) \in L^2_{hol}(A) \]
converges in $\ast$-distribution to a standard circular element $c$.  Then Define, for each $t>0$,
\[ \psi^{(N)}_t = \sum_{n\ge 0} e^{-nt} [a^{(N)}]^n \in L^2_{hol}(A). \]
Following the proof of Theorem \ref{Lp lower bound k>0}, we have $\|\psi_t^{(N)}\|_2 = (1-e^{-2t})^{1/2}$ (since $\|a^{(N)}\|_2 = 1$ for each $N$).  As before, $D_t\,\psi^{(N)}_t = \psi^{(N)}_{2t}$.  We also have that $\|\psi_{2t}^{(N)}\|_{2r}^{2r} = \p\left[\left(\psi_{2t}^{(N)}(\psi_{2t}^{(N)})^\ast\right)^r\right]$ converges, as $N\to\infty$, to $\p[(\psi_{2t}\psi_{2t}^\ast)^r]$ where $\psi_t = \sum_{n \ge 0} e^{-nt} c^n$.  (This follows by truncating the infinite sums in the definitions of $\psi_t^{(N)}$ and $\psi_t$, using the above limit-in-distribution, and then using the normality of $\p$.)  Appealing to Theorem \ref{Lp lower bound k>0}, since $\psi_t \in L^2_{hol}(c)$ for each $t>0$ and $c$ is $\mathscr{R}$-diagonal with non-negative cumulants, as $N\to\infty$ we have $\|\psi_{2t}^{(N)}\|_{2r}^{2r}$ converges to a limit which is $\ge \alpha_r\, t^{-3r+2}$ for some $\alpha_r>0$.  Now following the conclusion of the proof of Theorem \ref{Lp lower bound k>0}, we have
\[ \lim_{N\to\infty} \frac{\|\psi_{2t}^{(N)}\|_{2r}}{\|\psi_t^{(N)}\|_2} \ge \alpha_r\, t^{-1+\frac{1}{r}}. \]
Since $\psi_t^{(N)}\in L^2_{hol}(A)$ for each $N$ and each $t>0$, the result now follows.
\end{proof}

\subsection{Optimal Ultracontractivity}\label{optimal ultracontractivity}

We now prove Theorem \ref{theorem 3}.  It is restated here as Theorem \ref{optimal ultra}, with conditions that appear slightly different from those state in Theorem \ref{theorem 3}; however, the two are equivalent, as is explained in Remark \ref{remark v(a)}.

\begin{theorem} \label{optimal ultra} Let $a_1,\ldots,a_d$ be $\ast$-free $\mathscr{R}$-diagonal operators such that $\sup \|a_j\|/\|a_j\|_2 < \infty$, and suppose that one of the generators $a=a_j$ satisfies $\|a\|_2 =1$ while $\|a\|_4>1$.  Then there are constants $\alpha,\beta>0$ so that, for $0<t<1$,
\[ \alpha\,t^{-1} \le \|D_t\colon L^2_{hol}(a_1,\ldots,a_d) \to W^\ast(a_1,\ldots,a_d)\| \le \beta\, t^{-1}. \]
Moreover, this optimal lower bound is achieved on the subspace $L^2_{hol}(a)\subset L^2_{hol}(a_1,\ldots,a_d)$. \end{theorem}

\begin{remark} The upper bound of Theorem \ref{optimal ultra} was proved as Theorem 5.4 in \cite{Kemp Speicher}, stated in the special case that all the generators are identically-distributed.  In fact, the proof therein yields the above-stated theorem without modification, and so we only prove the lower bound below. \end{remark}

\begin{proof} The main technique we use here is the following simple estimate.  If $x$ is a bounded operator and $\ff$ is a state on $C^\ast(x)$, then since $xx^\ast$ is a positive semidefinite operator it is less than or equal to $\|xx^\ast\|\,1$; that is, $\|xx^\ast\|\,1 - xx^\ast$ is positive semidefinite.  A product of commuting positive semidefinite operators is positive semidefinite, and so $\|xx^\ast\|xx^\ast - xx^\ast xx^\ast \ge 0$.  Since $\ff$ is a state, it follows that
\[ \ff(xx^\ast xx^\ast) \le \ff( \|xx^\ast\|xx^\ast) = \|x\|^2\ff(xx^\ast). \]
In other words, we have the estimate $\|x\|^2 \ge \|x\|_4^4/\|x\|_2^2$.

\medskip

\noindent Now, let $a$ be $\mathscr{R}$-diagonal with $\|a\|_2=1$ and $\|a\|_4>1$, and for fixed $t>0$ let $x =D_t\,\psi_t$ from the proof of Theorem \ref{Lp lower bound k>0}:
\[ \psi_t = \sum_{n\ge 0} e^{-nt} a^n. \]
Then $x = D_t\,\psi_t = \psi_{2t}$.  We need to estimate $\|\psi_{2t}\|_4^4$ from below, but the estimate of Equation \ref{eqn expand psi_t 9} will not suffice because it holds valid only for $a$ with non-negative free cumulants; this is because we ignored non-pairing cumulants to develop it.  Here, instead, we will estimate this norm more accurately with all relevant partitions, using Theorem \ref{2 runs theorem}. 

\medskip

\noindent Equations \ref{eqn expand psi_t 3} and \ref{eqn lower bound cumulant expansion 1}, in the present case $r=2$, say
\begin{equation} \label{eqn 4 norm 1}
\|\psi_{2t}\|_{4}^{4}= \sum_{n\ge 0}  e^{-4nt} \sum_{n_1+n_2 = m_1+m_2=n}  \ff(a^{n_1} a^{\ast m_1}a^{n_2} a^{\ast m_2}),
\end{equation}
where
\begin{equation} \ff(a^{n_1} a^{\ast m_1}a^{n_2} a^{\ast m_2}) = \sum_{\pi\in NC(1^{n_1},\ast^{m_1},1^{n_2},\ast^{m_2})} \k_\pi[a^{,n_1},a^{\ast,m_1},a^{,n_2},a^{\ast,m_2}].
\end{equation}
Break up this sum into those $\pi$ that contain a $4$-block, and those that are pairings.  From the proof of Theorem \ref{2 runs theorem}, if $\pi$ contains a $4$-block then it is the only one, and all other blocks are pairings.  We normalized $\|a\|_2=1$, so that $\k_2[a,a^\ast] = \k_2[a^\ast,a] =1$.  The four block yields a term $\k_4[a,a^\ast,a,a^\ast]$ or $\k_4[a^\ast,a,a^\ast,a]$.  From page 177 in \cite{Nica Speicher Book} we calculate these as
\[ \k_4[a,a^\ast,a,a^\ast] = \k_4[a^\ast,a,a^\ast,a]  = \ff(aa^\ast aa^\ast) - 2\ff(aa^\ast)^2 = \|a\|_4^4 - 2. \]
Let $v(a) = \|a\|_4^4-1$.  Thus, if $\pi$ contains a $4$-block, then $\k_\pi[a^{,n_1},a^{\ast,m_1},a^{,n_2},a^{\ast,m_2}] = v(a)-1$, while if $\pi\in NC_2$ then $\k_\pi[a^{,n_1},a^{\ast,m_1},a^{,n_2},a^{\ast,m_2}]  = 1$.  From the enumeration of $NC(\mx{S})$ and $NC_2(\mx{S})$ in Theorem \ref{2 runs theorem}, we then have
\[ \ff(a^{n_1}a^{\ast m_1}a^{n_2}a^{\ast m_2}) = \mu\,(v(a)-1) + \mu+1 = \mu\cdot v(a) + 1, \]
where $\mu = \min\{n_1,m_1,n_2,m_2\}$.  Note that the assumption $\|a\|_4>1$ is precisely to ensure that $v(a)>0$ in this expression.  From Equation \ref{eqn 4 norm 1}, we have
\begin{equation} \label{eqn 4 norm 2}
\|\psi_{2t}\|_{4}^{4}= \sum_{n\ge 0}  e^{-4nt} \sum_{n_1+n_2 = m_1+m_2=n} \left(1 + v(a)\min\{n_1,m_1,n_2,m_2\}\right).
\end{equation}
Rewrite the internal sum in terms of only the variables $n_1=i,m_1=j$,
\[ \sum_{0\le i,j\le n} \left(1+v(a)\min\{i,j,n-i,n-j\}\right) = (n+1)^2 + v(a)\sum_{0\le i,j\le n} \min\{i,j,n-i,n-j\}. \]
This sum can be evaluated exactly, but for this estimate it is sufficient to look only at the terms
\[ \sum_{0\le i,j\le n} \min\{i,j,n-i,n-j\} \ge \sum_{0\le i\le j\le n/2} \min\{i,j,n-i,n-j\}. \]
Since $i\le j$, $n-j\le n-i$, and so the summation is over $\min\{i,n-j\}$.  We can then write this as $\sum_{j=0}^{n/2} \sum_{i=0}^j \min\{i,n-j\}$, and since $j\le n/2$, $n-j\ge n/2 \ge i$ so we have
\[\begin{aligned} \sum_{j=0}^{n/2} \sum_{i=0}^j \min\{i,n-j\} \;  &= \sum_{0\le i\le j\le n/2} \min\{i,n-j\} = \sum_{j=0}^{n/2} \sum_{i=0}^j i \\
&= \frac{1}{2}\sum_{j=0}^{n/2} j(j+1) = \frac{1}{48} n^3 + \frac{1}{4} n^2 + \frac{1}{3}n.
\end{aligned} \]
Returning to Equation \ref{eqn 4 norm 2}, this means
\begin{equation} \label{eqn 4 norm 3}
\|\psi_{2t}\|_4^4 \ge \frac{1}{48}v(a)\sum_{n\ge 0} n^3\,e^{-4nt} \ge \alpha\,v(a)\, t^{-4},
\end{equation}
where the second inequality follows from Lemma \ref{lemma sum exp}.  Now, Equation \ref{2 norm a} yields (via the normalization $\|a\|_2 =1$ and replacing $t$ with $2t$) $\|\psi_{2t}\|_2^2 = \sum_{n\ge 0}e^{-4nt} \le \alpha\,t^{-1}$, and so from the discussion at the beginning of the proof, $\|\psi_{2t}\|^2 \ge \|\psi_{2t}\|_4^4/\|\psi_{2t}\|_2^2 \ge \alpha\,v(a)\, t^{-3}$.  Finally, again using Equation \ref{2 norm a} we have $\|\psi_t\|_2^2 \ge \alpha\,t^{-1}$, and thus
\[ \frac{\| D_t\,\psi_t \|^2}{\|\psi_t\|_2^2} \ge \alpha\,v(a)\,t^{-2}. \]
Since $v(a)>0$, this proves the result.
\end{proof}

\begin{remark} The condition $\|a\|_4>1$ (equivalently $v(a) = \|a\|_4^4-1>0$) is required for the theorem to hold.  Indeed, if $v(a)=0$ then Equation \ref{eqn 4 norm 2} gives $\|\psi_{2t}\|_4^4=\sum_{n} e^{-4nt} (n+1)^2$ which is of order $t^{-3}$; this translates to an $O(t^{-1/2})$--lower--bound for the action of $D_t$, as is realized by a Haar unitary generator.
\end{remark}

\begin{remark}\label{remark v(a)} In fact, the condition $v(a)=0$ actually {\em requires} $a$ to be (a scalar multiple of a) Haar unitary.  For if $\|a\|_2=1$ and $v(a)=0$ then $\|a\|_4 = 1$ as well.  Let $b = a^\ast a$, so $\ff(b) = \|a\|_2^2 = 1$ and $\ff(b^2) = \|a\|_4^4=1$.  Since $b\ge 0$, this means that $\mathrm{Var}(b) = \ff(b^2)-\ff(b)^2 = 0$, and therefore $b$ is a.s.\ constant; since $\ff(b)=1$, this means $b=1$ a.s.  Now write $a=ur$ with $u$ Haar unitary and $r\ge 0$, following Remark \ref{remark polar decomp}.  Then $b=a^\ast a = r^2$ and so $r^2 = 1$ a.s.\ which implies that $r=1$ a.s.\ since $r\ge 0$.  Thus $a=u$, a Haar unitary as required.  Hence, Theorem \ref{optimal ultra} shows {\em universal behaviour} for all non-Haar unitary $\mathscr{R}$-diagonal operators.
\end{remark}

\begin{remark} Letting $e^{-t}=1/\l$, the proofs of Theorems \ref{Lp lower bound k>0} and \ref{optimal ultra} are estimates for various norms of the resolvent of $a$: $\psi_t = \sum_{n\ge 0} e^{-nt} a^n = \l/(\l-a)$.  We have provided here only rough estimates of constants involved.  In fact, there is completely universal behaviour for this resolvent function over all $\mathscr{R}$-diagonal $a$; the only dependence in the blow-up is on $v(a)$.  These sharp estimates will be discussed in \cite{HKS}. \end{remark}

\bigskip

\noindent {\bf Acknowledgments.} The author wishes to thank Uffe Haagerup, Karl Mahlburg, Mark Meckes, and Luke Rogers for useful conversations.


\begin{thebibliography}{37}

\bibitem{Wlodek 1} Bernau, S.; Lacey, H.: {\em The range of a contractive projection on an $L\sb{p}$-space.}  Pacific J. Math. {\bf 53}, 21--41 (1974)

\bibitem{BD} Beurling, A.; Deny, J.: \emph{Espaces de Dirichlet. I. Le cas \'el\'ementaire.}  Acta Math. {\bf 99}, 203--224 (1958)

\bibitem{Biane 1} Biane, P.: {\em Segal-Bargmann transform, functional calculus on matrix spaces and the theory of semi-circular and circular systems.} J. Funct. Anal. {\bf 144}, 232--286 (1997)

\bibitem{Biane 2} Biane, P.: \emph{Free hypercontractivity.} Commun. Math. Phys. {\bf 184}, 457-474 (1997)

\bibitem{Blanchard Dykema} Blanchard, E.; Dykema, K.: {\em Embeddings of reduced free products of operator algebras.} Pacific J. Math. {\bf 199}, no. 1, 1--19 (2001)

\bibitem{BoKS} Bozejko, M., K\"ummerer, B., Speicher, R.: {\em q-Gaussian processes: non-commutative and classical aspects.} Commun. Math. Phys. {\bf 185}, 129-154 (1997)

\bibitem{REU} Chou, E.; Fricano, A.; Kemp, T.; Poh, J.; Shore, W.; Whieldon, G.; Wong, T.; Zhang, Y.: {\em Convex posets in non-crossing pairings of bitstrings.}  Preprint.

\bibitem{GKLZ} Graczyk, P.; Kemp, T.; Loeb, J.; \.Zak, T.: {\em Hypercontractivity for log-subharmonic functions.} Preprint.

\bibitem{Gross 3} Gross, L.: {\em Hypercontractivity over complex manifolds.} Acta. Math. {\bf 182}, 159-206 (1999)

\bibitem{Haagerup 2} Haagerup, U.: {\em Random matrices, free probability and the invariant subspace problem relative to a von Neumann algebra.} Proceedings of the International Congress of Mathematicians, Vol. I (Beijing, 2002),  273--290

\bibitem{HKS} Haagerup, U.; Kemp, T.; Speicher, R.: {\em Resolvents and norms for $\mathscr{R}$-diagonal operators.}  Preprint.

\bibitem{Haagerup Larsen} Haagerup, U.; Larsen, F.: {\em Brown's spectral distribution measure for $\mathscr{R}$-diagonal elements in finite von Neumann algebras.} J. Funct. Anal. {\bf 176}, 331-367 (2000)

\bibitem{Janson} Janson, S.: {\em On hypercontractivity for multipliers on orthogonal polynomials.} Ark. Math. {\bf 21}, 97-110 (1983)

\bibitem{JLX} Junge, M.; Le Merdy, C.; Xu, Q.: {\em Calcul fonctionnel et fonctions carr\'ees dans les espaces $L\sp p$ non commutatifs.} C. R. Math. Acad. Sci. Paris {\bf 337} 93-98 (2003)

\bibitem{JX} Junge, M.; Xu, Q.: {\em Th\'eor\`emes ergodiques maximaux dans les espaces $L\sb p$ non commutatifs.}  C. R. Math. Acad. Sci. Paris {\bf 334} 773-778 (2002) 

\bibitem{Kemp} Kemp, T.: {\em Hypercontractivity in non-commutative holomorphic spaces.} Commun. Math. Phys. {\bf 259}, 615-637 (2005)

\bibitem{KMRS} Kemp,T.; Mahlburg, K.; Smyth, C.; Rattan, A.: {\em Enumeration of non-crossing pairings on bitstrings.}  Preprint.

\bibitem{Kemp Speicher} Kemp, T.; Speicher, R.: {\em Strong Haagerup inequalities for free $\mathscr{R}$-diagonal elements.}    J. Funct. Anal.  251  (2007),  no. 1, 141--173.

\bibitem{Wlodek 2} Moy, S.: {\em Characterizations of conditional expectation as a transformation on function spaces.}
Pacific J. Math. {\bf 4}, 47--63 (1954)

\bibitem{NSS} Nica, A.; Shlyakhtenko, D.; Speicher, R.: {\em Maximality of the microstates free entropy for $\mathscr{R}$-diagonal elements.}  Pacific J. Math.  187 no. {\bf 2}, 333--347 (1999)

\bibitem{Nica Speicher Fields Paper} Nica, A.; Speicher, R.: {\em $\mathscr{R}$-diagonal pairs---a common approach to Haar unitaries and circular elements.} Fields Inst. Commun., {\bf 12}, 149-188 (1997)

\bibitem{Nica Speicher Duke Paper} Nica, A.; Speicher, R.: {\em Commutators of free random variables.} Duke Math. J. {\bf 92}, 553-592 (1998)

\bibitem{Nica Speicher Book} Nica, A.; Speicher, R.: {\em Lectures on the Combinatorics of Free Probability.}  London Mathematical Society Lecture Note Series, no. 335, Cambridge University Press, 2006

\bibitem{Wlodek 3} Olson, M.: {\em A characterization of conditional probability.} Pacific J. Math. {\bf 15}, 971--983 (1965) 

\bibitem{Speicher 1} Speicher, R.: {\em A new example of independence and white noise.} Probab.\ Th.\ Rel.\ Fields {\bf 84}, 141-159 (1990)

\bibitem{Sniady Speicher} \'Sniady, P.; Speicher, R.: {\em Continuous family of invariant subspaces for $\mathscr{R}$-diagonal operators.} Invent. Math. {\bf 146}, no. 2, 329--363 (2001)


\bibitem{Wlodek 4} Wulbert, D.: {\em A note on the characterization of conditional expectation operators.}  Pacific J. Math. {\bf 34}, 285--288 (1970)


\end{thebibliography}
\end{document}